\documentclass[11pt]{article}
\pdfoutput=1

\usepackage[margin=1in]{geometry}

\usepackage{amsfonts} 
\newcommand{\norm}[1]{\left\lVert#1\right\rVert}
\usepackage{bm}%
\usepackage{graphicx}

\usepackage{amsmath}
\usepackage{dsfont}     
\graphicspath{{figures/}}

\usepackage{amsthm}
\usepackage{cite}
\usepackage{hyperref}

\usepackage[nameinlink]{cleveref}

\theoremstyle{definition}
\newtheorem{theorem}{Theorem}

\newtheorem{lemma}{Lemma}

\newtheorem{remark}{Remark}

\begin{document}

\title{HDG-POD Reduced Order Model of the Heat Equation}

\author{Jiguang Shen%
	\thanks{School  of Mathematics, University of Minnesota, MN (shenx179@umn.edu)}%
	\and
	John~R.~Singler%
	\thanks{Department of Mathematics
		and Statistics, Missouri University of Science and Technology,
		Rolla, MO (\mbox{singlerj@mst.edu}, ywzfg4@mst.edu). J.~Singler and Y.~Zhang were supported in part by National Science Foundation grant DMS-1217122.  J.~Singler and Y.~Zhang thank the IMA for funding research visits, during which some of this work was completed.}
	\and
	Yangwen Zhang%
	\footnotemark[2]
}

\date{}

\maketitle

\begin{abstract}
	We propose a new hybridizable discontinuous Galerkin (HDG) model order reduction technique based on proper orthogonal decomposition (POD).  We consider the heat equation as a test problem and prove error bounds that converge to zero as the number of POD modes increases.  We present 2D and 3D numerical results to illustrate the convergence analysis.
\end{abstract}


\section{Introduction}
Discontinuous Galerkin (DG) methods for solving partial differential equations, developed in the late 1990s, have become popular among computational scientists \cite{arnold2002unified, arnold2000discontinuous, bassi1997high, castillo2000priori, cockburn1998local}. However, the number of degrees of freedom for DG methods is much larger compared to standard continuous Galerkin (CG) methods; this is typically considered to be the main drawback of DG methods.

Hybridizable discontinuous Galerkin (HDG) methods were originally proposed by Cockburn, Gopalakrishnan, and Lazarov in \cite{cockburn2009unified} to improve traditional DG methods.  HDG methods keep the positive features of DG methods, while simultaneously lowering the number of degrees of freedom.  In addition to approximating a scalar variable and its flux, HDG methods also approximate the trace of the scalar variable on the boundary of the mesh.  The scalar and flux variables are eliminated element-by-element, which leaves a global system in terms of the approximate scalar trace only.  This greatly reduces the number of degrees of freedom compared to other DG methods.  Because of the many advantages of HDG methods, they are being actively investigated in many directions; see, e.g.,  \cite{nguyen2009implicit1,nguyen2009implicit2,cockburn2009hybridizable,nguyen2010hybridizable,nguyen2011implicit,nguyen2011high,nguyen2011hybridizable,nguyen2012hybridizable,huynh2013high,rhebergen2012space,ueckermann2016hybridizable,Bui-Thanh16,Muralikrishnan17,GaticaSequeira16,GaticaSequeira17,Oikawa16,CuiZhang14}.

Although the degrees of freedom for HDG methods is much smaller compared to other DG methods, the resulting discrete systems can have large dimensions for complex applications.  Therefore, efficient and accurate model order reduction techniques are desirable.

Proper orthogonal decomposition (POD) is a popular model reduction technique that has been successfully used in a variety of fields including fluid dynamics \cite{holmes2012turbulence}, inverse problems \cite{banks2000nondestructive}, PDE-constrained optimization \cite{sachs2010pod,gubisch2013proper}, and feedback control \cite{atwell2001reduced,kunisch2004hjb,leibfritz2006reduced,ravindran2000reduced,LeeTran05}.  Many researchers have developed POD model reduction schemes for many problems and have proved related error bounds, see, e.g., \cite{volkwein2013proper,kunisch2001galerkin,kunisch2002galerkin,homescu2005error,luo2008mixed,singler2014new,iliescu2014variational,akman2016error,jin2017analysis,chapelle2012galerkin,amsallem2014error,herkt2013convergence,del2008error,chapelle2013galerkin,Kostova-Vassilevska18,MohebujjamanRebholzXieIliescu17,XieWellsWangIliescu18}.

A POD reduced order model of a time dependent partial differential equation (PDE) is typically constructed in the following way.  First, solution data is collected and a POD procedure is used to identify a small or moderate number of global basis functions, called POD modes, that optimally approximate the data.  These global basis functions are then used in combination with a standard Galerkin method (or a Petrov-Galerkin method) to derive the reduced order model (ROM).

Since HDG methods have many excellent properties and are increasingly utilized to simulate a wide range of complex problems, in this work we propose using POD in combination with the HDG variational form to derive reduced order models of time dependent PDEs.  This HDG-POD model order reduction approach can be applied directly with the same HDG variational formulation used for the HDG simulation; we do not require the use of POD with a different Galerkin weak form to generate the ROM.

We begin our investigation of the HDG-POD model order reduction approach with a test problem: the heat equation.  Using POD with HDG is not a straightforward combination of existing methods since the HDG variational formulation is very different than the Galerkin approaches typically used to construct POD-based ROMs.  Furthermore, the standard HDG semidiscretization of the heat equation does not yield an ordinary differential equation for the spatially discrete scalar variable.  However, we are able to use POD and the HDG weak form to construct a dynamic HDG-POD reduced order model for the scalar variable in \Cref{sec:HDG_POD_ROM}.

HDG methods use a spatial mixed formulation, which yields higher accuracy for the flux variable.  For numerical methods not based on a mixed formulation, lower order convergence rates for the flux are common.  Reduced order models constructed using these types of numerical methods can potentially require a fine spatial mesh to yield highly accurate flux approximations.  In contrast, HDG generates high order convergence rates for both the scalar and flux variables simultaneously.  In the HDG-POD reduced model, we can easily recover flux approximations at computational cost $\mathcal O(r)$, where $r$ is the order of reduced model.

In \Cref{sec:error_analysis}, we analyze the HDG-POD reduced order model of the heat equation and prove error bounds that tend to zero as the order of the reduced model increases.  As mentioned above, similar convergence analyses have been performed for many standard and modified POD reduced order models.  However, since HDG-POD is not based on a standard Galerkin or Petrov-Galerin projection, the error analysis performed here is quite different from existing POD analysis work.  We present 2D and 3D numerical results illustrating the theory in \Cref{sec:numerical_results}.

We emphasize that the heat equation is only a test problem for this new HDG-POD model order reduction method.  There are many other approaches that can be used to construct ROMs for the heat equation.  We do not claim that HDG-POD can produce a better ROM for the heat equation compared to other approaches.  Again, our goal is to develop the HDG-POD approach to use in combination with HDG methods to devise and analyze ROMs for complex nonlinear PDEs, such as the incompressible Navier-Stokes equations with high Reynolds number.  The consideration of HDG-POD for the heat equation is a first step towards this goal.

\section{HDG-POD Model Order Reduction}
\label{sec:HDG_POD_ROM}

We introduce our HDG-POD model order reduction procedure for the heat equation step by step.  We begin by reviewing an HDG method for the heat equation in \Cref{subsec:HDG} and then we review the POD data approximation problem in \Cref{sec:POD}.  We give the formulation of the HDG-POD reduced order model in continuous time in \Cref{sec:HDG-POD_details}.

Let $\Omega$ be a bounded domain in $\mathbb{R}^d $ $ (d=2,3)$ with Lipschitz boundary $\Gamma$, and let $T>0$.  We consider the following heat equation
\begin{equation}\label{heat}
\begin{split}
u_t - \nabla\cdot(a\nabla u) &=f \quad  \quad  \;\;\;\;\;  \text{in} \; \Omega,\\
u&=0 \quad \quad \quad\; \text{on}\;  \Gamma,\\
u(\cdot,0)& = u_0   \quad \quad \quad \text{in}\;  \Omega, 
\end{split}
\end{equation}
where $f\in L^2(0,T;L^2(\Omega))$ is the source function, $u_0\in L^2(\Omega)$ is the initial condition, and $a\in L^{\infty}(\Omega)$ is the coefficient function.  We assume there exist positive constants $c_0$ and $c_1$ such that $0<c_0\le a^{-1} \le c_1$.

\subsection{HDG}
\label{subsec:HDG}

Next, we briefly describe an HDG method for the heat equation, which was analyzed in \cite{chabaud2012uniform}.  Let $\mathcal{T}_h$ be a collection of disjoint elements that partition $\Omega$.  We let $\partial \mathcal{T}_h$ denote the set $\{\partial K: K\in \mathcal{T}_h\}$. For an element $K$ of the collection  $\mathcal{T}_h$, $e = \partial K \cap \Gamma$ is a boundary face if the $d-1$ Lebesgue measure of $e$ is nonzero. For two elements $K^+$ and $K^-$ of the collection $\mathcal{T}_h$, $e = \partial K^+ \cap \partial K^-$ is the interior face between $K^+$ and $K^-$ if the $d-1$ Lebesgue measure of $e$ is nonzero. Let  $\varepsilon_h^{\partial}$ and $\varepsilon_h^{o}$  denote the set of boundary faces and interior faces, respectively, and we let $\varepsilon_h$ denote the union of  $\varepsilon_h^o$ and $\varepsilon_h^{\partial}$. We finally introduce
\begin{align*}
	(w,v)_{\mathcal{T}_h} = \sum_{K\in\mathcal{T}_h} (w,v)_K,   \quad\quad\left\langle \zeta,\rho\right\rangle_{\partial\mathcal{T}_h} = \sum_{K\in\mathcal{T}_h} \left\langle \zeta,\rho\right\rangle_{\partial K}.
\end{align*}

\subsubsection{The HDG Formulation: Discretization in Space}

The HDG method introduces the flux $\bm{q} = -a\nabla u$ as an unknown.  For $c = a^{-1}$, the heat equation \eqref{heat} can be rewritten as 
\begin{subequations}\label{HDG_pde}
	\begin{align}
		(c\bm{q},\bm{v})_{\mathcal{T}_h}-(u,\nabla\cdot\bm{v})_{\mathcal{T}_h}+\left\langle {u}, \bm{v}\cdot \bm{n}\right\rangle_{\partial {\mathcal{T}_h}} &=0, \label{HDG_pde_a}\\
		(u_t, w)_{\mathcal{T}_h}-(\bm{q}, \nabla w)_{\mathcal{T}_h}+\left\langle {\bm{q}}\cdot \bm{n}, w\right\rangle_{\partial {\mathcal{T}_h}} &= (f,w)_{\mathcal{T}_h},  \label{HDG_pde_b}\\
		\left\langle {\bm{q}}\cdot \bm{n}, \mu \right\rangle_{\partial\mathcal{T}_h\backslash \varepsilon_h^{\partial}} &=0, \label{HDG_pde_c}\\
		\left\langle u, \mu \right\rangle_{\varepsilon_h^{\partial}} &=0, \label{HDG_pde_d}\\
		(u(\cdot,0),w)_{\mathcal{T}_h} &= (u_0,w)_{\mathcal{T}_h}, \label{HDG_pde_e}
	\end{align}
\end{subequations}
for all $(\bm{v},w,\mu)\in H(\text{div},\Omega)\times L^2(\Omega)\times  L^2(\varepsilon_h)$.

Let $\mathcal{P}^k(D)$ denote the set of polynomials of degree at most $k$ on a domain $D$.  We introduce the discontinuous finite element spaces
\begin{equation}
\begin{split}
\bm{V}_h  &:= \{\bm{v}\in [L^2(\Omega)]^d: \bm{v}|_{K}\in [\mathcal{P}^k(K)]^d, \forall K\in \mathcal{T}_h\},\\
{W}_h  &:= \{{w}\in L^2(\Omega): {w}|_{K}\in \mathcal{P}^k(K), \forall K\in \mathcal{T}_h\},\\
{M}_h  &:= \{{\mu}\in L^2(\mathcal{\varepsilon}_h): {\mu}|_{e}\in \mathcal{P}^k(e), \forall e\in \varepsilon_h, \mu|_{\varepsilon_h^{\partial}} =0\}.
\end{split}
\end{equation}
Note that $M_h$ consists of functions that are continuous inside the faces (or edges) $e\in \varepsilon_h$ and discontinuous at their borders.

The HDG method seeks an approximation  $({\bm{q}}_h,u_h,\widehat u_h)\in \bm{V}_h\times W_h\times M_h$ such that
\begin{subequations}\label{HDG_discrete2}
	\begin{align}
		(c\bm{q}_h, \bm{v})_{{\mathcal{T}_h}}- (u_h, \nabla\cdot \bm{v})_{{\mathcal{T}_h}}+\langle \widehat u_h, \bm{v}\cdot \bm{n} \rangle_{\partial{{\mathcal{T}_h}}} &=0, \label{HDG_discrete2_a}\\
		(\partial_t{u_h}, w)_{{\mathcal{T}_h}}-(\bm{q}_h, \nabla w)_{{\mathcal{T}_h}}
		+\langle\widehat{\bm{q}}_h\cdot\bm{n}, w \rangle_{\partial {{\mathcal{T}_h}}}&=(f, w)_{{\mathcal{T}_h}}, \label{HDG_discrete2_b}\\
		\langle\widehat{\bm{q}}_h\cdot \bm{n}, \mu \rangle_{\partial\mathcal{T}_{h}\backslash {{\varepsilon_h^{\partial}}}}&=0, \label{HDG_discrete2_c}\\
		(u_h(\cdot,0),w)_{\mathcal{T}_h} &= (u_0,w)_{\mathcal{T}_h}, \label{HDG_discrete2_d}
	\end{align}
	for all $(\bm{v},w,\mu)\in \bm{V}_h\times W_h\times M_h$. Here, the numerical trace $ \widehat{\bm{q}}_h $ on $\partial\mathcal{T}_h$ is defined by 
	\begin{align}
		\widehat{\bm{q}}_h &= {\bm{q}_h}+\tau(u_h-\widehat{{u}}_h)\bm{n}  \quad \mbox{on} \; \partial \mathcal{T}_h, \label{HDG_discrete2_e}
	\end{align}
\end{subequations}
where $\tau$ is  positive stabilization function defined on $\partial\mathcal{T}_h$, which we assume to be constant on each edge of a triangle or each face of a  tetrahedron.

\subsubsection{HDG Implementation: Local Solver and Time Discretization}
\label{sec:HDG_local_solver}

After some simple manipulations with \eqref{HDG_discrete2_a}-\eqref{HDG_discrete2_e}, it can be seen that 
$(\bm q_h,u_h,\widehat u_h)$ is the solution of the following weak formulation:
\begin{subequations}\label{Heat_HDG}
	\begin{align}
		(c\bm{q}_h,\bm{v})_{\mathcal{T}_h}-(u_h,\nabla\cdot \bm{v})_{\mathcal{T}_h}+\left\langle\widehat{u}_h,\bm{v\cdot n} \right\rangle_{\partial {\mathcal{T}_h}}&= 0, \label{Heat_HDG_a}\\
		(\partial_tu_h,w)_{\mathcal T_h}+(\nabla \cdot \bm{q}_h, w)_{\mathcal{T}_h}+\left\langle\tau(u_h-\widehat u_h),w\right\rangle_{\partial {\mathcal{T}_h}} &= (f,w)_{\mathcal{T}_h},\label{Heat_HDG_b}\\
		\left\langle \bm q_h\cdot \bm n +\tau(u_h-\widehat u_h), \mu\right\rangle_{\partial{\mathcal{T}_h}\backslash\varepsilon_h^{\partial}} &=0,\label{Heat_HDG_c}\\
		(u_h(\cdot,0),w)_{\mathcal{T}_h} &=(u_0,w)_{\mathcal{T}_h} ,\label{Heat_HDG_d}
	\end{align}
\end{subequations}
for all  $({\bm{v}},w,\mu)\in \bm{V}_h \times W_h\times M_h$.

Assume $\bm{V}_h = \mbox{span}\{\bm\varphi_i\}_{i=1}^{N_1}$, $W_h=\mbox{span}\{\phi_i\}_{i=1}^{N_2}$, $M_h=\mbox{span}\{\psi_i\}_{i=1}^{N_3}$, and write the unknowns as
\begin{align}\label{expre}
	&\bm{q}_h=  \sum_{j=1}^{N_1}\alpha_{j}(t)\bm\varphi_j,  \ 
	u_h = \sum_{j=1}^{N_2}\beta_{j}(t)\phi_k,    \ 
	\widehat{u}_h = \sum_{j=1}^{N_3}\gamma_{j}(t)\psi_{j}.
\end{align}
Also write the coefficient vectors as $\bm \alpha(t) = [\alpha_1(t),\ldots,\alpha_{N_1}(t)]^T$, $\bm \beta(t) = [\beta_1(t),\ldots,\beta_{N_2}(t)]^T$, and $\bm \gamma(t) = [\gamma_1(t),\ldots,\gamma_{N_3}(t)]^T$.

Substitute \eqref{expre} into \eqref{Heat_HDG_a}-\eqref{Heat_HDG_c} and use the corresponding test functions to test \eqref{Heat_HDG_a}-\eqref{Heat_HDG_c} respectively, to obtain the following matrix equation
\begin{subequations}\label{system_equation}
	\begin{align}
		\begin{bmatrix}
			0 &0&0\\
			0 &M&0\\
			0 &0&0
		\end{bmatrix}
		\left[ {\begin{array}{*{20}{c}}
				\bm{\alpha'}(t)\\
				\bm{\beta'}(t)\\
				\bm{\gamma'}(t)\\
		\end{array}} \right]
		+	 \begin{bmatrix}
			A_1 & -A_2 &A_3\\
			A_2^T & A_4 &-A_5\\
			A_3^T & A_5^T &-A_6
		\end{bmatrix}
		\left[ {\begin{array}{*{20}{c}}
				\bm{\alpha}(t)\\
				\bm{\beta}(t)\\
				\bm{\gamma}(t)
		\end{array}} \right]
		&=\left[ {\begin{array}{*{20}{c}}
				0\\
				b_1(t)\\
				0\\
		\end{array}} \right],
		\\\nonumber
		\\
		\bm\beta(0) &= \bm{\beta}_0,
	\end{align}
\end{subequations}
where
\begin{align*}
	&A_1= [(c\bm\varphi_j,\bm\varphi_i )_{\mathcal{T}_h}],\qquad A_2 = [(\phi_j,\nabla\cdot\bm{\varphi}_i)_{\mathcal{T}_h}], \qquad A_3 = [\left\langle\psi_j,{\bm\varphi_i\cdot\bm n}\right\rangle_{\partial\mathcal{T}_h}],
	\\
	&A_4=
	[\left\langle\tau\phi_j,{\phi_i}\right\rangle_{\partial\mathcal{T}_h}],\qquad A_5 = [\left\langle\tau\psi_j,{\varphi_i}\right\rangle_{\partial\mathcal{T}_h}], \qquad\;\; A_6 = [\left\langle\tau\psi_j,{\psi_i}\right\rangle_{\partial\mathcal{T}_h}] ,\\
	&M= [(\phi_j,\phi_i )_{\mathcal{T}_h}],\qquad \qquad b_1(t) = [(f,\phi_i )_{\mathcal{T}_h}],\qquad\;\;  b_2 = [(u_0,\phi_i )_{\mathcal{T}_h}],
\end{align*}
and $\bm{\beta}_0$ is determined by solving the linear system
\begin{align*}
	M\bm{\beta}_0 = b_2.
\end{align*}

A finite difference method can be applied to solve \eqref{system_equation}, but the computation is costly since we have three variables.  The local solver in the HDG method avoids this by eliminating $\bm \alpha$ and $\bm\beta$.

We now turn our attention to the implementation of the local solver in combination with a simple scheme for discretization with respect to the time variable.  Other time stepping methods can be handled similarly.  We introduce a time step $\Delta t$, the time levels $t_n= n\Delta t$, and approximations $(\bm \alpha^{n},\bm \beta^{n},\bm \gamma^{n}) $ of the exact solution $(\bm\alpha, \bm{\beta}, \bm{\gamma})$ evaluated at time $ t_n $. We apply Backward Euler method to discretize the time derivative in the system \eqref{system_equation}, which leads to the linear system at each time level:
\begin{align}\label{Heat_algebra}
	\begin{bmatrix}
		A_1 & -A_2 &A_3\\
		A_2^T & \Delta t^{-1} M+A_4 &-A_5\\
		A_3^T & A_5^T &-A_6
	\end{bmatrix}
	\left[ {\begin{array}{*{20}{c}}
			\bm{\alpha}^{n}\\
			\bm{\beta}^{n}\\
			\bm{\gamma}^{n}
	\end{array}} \right]
	=\left[ {\begin{array}{*{20}{c}}
			0\\
			b_{1,n} +\Delta t^{-1} M \bm{\beta}^{n-1} \\
			0\\
	\end{array}} \right],
\end{align} 
i.e.,
\begin{subequations}\label{algebre_full}
	\begin{align}
		A_1 \bm\alpha^{n} -A_2\bm\beta^{n} +A_3\bm\gamma^{n} &= 0,\label{algebre_full_a}\\
		A_2^T \bm\alpha^{n} + \left(\Delta t^{-1} M+A_4\right)\bm\beta^{n} -A_5\bm\gamma^{n} &= b_{1,n}+\Delta t^{-1} M \bm{\beta}^{n-1},\label{algebre_full_b}\\
		A_3^T \bm\alpha^{n} + A_5^T\bm\beta^{n} -A_6\bm\gamma^{n} &= 0.\label{algebre_full_c}
	\end{align}
\end{subequations}

The HDG local solver takes advantage of the structure induced by the discontinuous finite element spaces to eliminate $ \bm \alpha^{n} $ and $ \bm \beta^{n} $ from the above linear system.  Specifically, because of the discontinuous nature of the spaces $ \bm V_h $ and $ W_h $, the matrices $ M $, $ A_1 $, and $ A_4 $ are all block diagonal with small blocks.  Therefore, these matrices can be easily inverted.

To apply this observation, use \eqref{algebre_full_a} and \eqref{algebre_full_b} to express $\bm\alpha^n, \bm \beta^n$ in terms of $\bm{\gamma}^n$ as follows:
\begin{align*}
	\bm{\alpha}^n &= A_1^{-1}A_2 Q^{-1}\left((A_5+A_2^TA_1^{-1}A_3)\bm{\gamma}^n+ b_{1,n}+\Delta t^{-1} M\bm{\beta}^{n-1} \right)-A_1^{-1}A_3\bm{\gamma}^n\\
	&=:\tilde A_1 \bm{\gamma}^n +\tilde b_1,\\
	\bm{\beta}^n &= Q^{-1}\left((A_5+A_2^TA_1^{-1}A_3)\bm{\gamma}^n+ b_{1,n}+\Delta t^{-1} M\bm{\beta}^{n-1} \right)\\
	&=:\tilde A_2 \bm{\gamma}^n +\tilde b_2,
\end{align*}
where
$$
Q := A_2^TA_1^{-1}A_2+\Delta t^{-1} M +A_4.
$$
Then insert $\bm{\alpha}^n$ and $\bm{\beta}^n$ into \eqref{algebre_full_c} to obtain the final linear system only involving $\bm{\gamma}^{n}$:
\begin{align}
	(A_3^T \tilde A_1 + A_5^T \tilde A_2 - A_6) \bm{\gamma}^n = -A_3^T\tilde b_1 -A_5^T \tilde b_2.
\end{align}

In the above computation, we need to compute the inverses of the matrices
$A_1$ and $Q$.  As mentioned above, $ A_1 $ is block diagonal with small blocks and therefore it is easy to invert.  Furthermore, $ A_1^{-1} $ is also block diagonal with small blocks.  For the matrix $ Q $, note that $A_2$ is not block diagonal but there exist matrices $A_{21},A_{22}$ in 2D or $A_{21},A_{22}, A_{23}$ in 3D such that all matrices are block diagonal with small blocks and
\begin{align*}
	A_2 = \left[ {\begin{array}{*{20}{c}}
			A_{21}\\
			A_{22}
	\end{array}} \right] \;\text{in 2D, or}\;
	A_2 = \left[ {\begin{array}{*{20}{c}}
			A_{21}\\
			A_{22}\\
			A_{23}
	\end{array}} \right] \;\text{in 3D}.
\end{align*}
Therefore, since $ M $, $ A_1^{-1} $, and $ A_4 $ are block diagonal with small blocks, $ Q = A_2^TA_1^{-1}A_2+\Delta t^{-1} M+A_4$ is also block diagonal with small blocks and is easily inverted.

The above  process is equivalent to the local solver defined in \cite{cockburn2009unified}. For more details about the HDG method, see  
\cite{nguyen2009implicit1,nguyen2009implicit2,nguyen2010hybridizable} and the references therein.

\subsection{POD Data Approximation Problem}
\label{sec:POD}

Next, we briefly review the proper orthogonal decomposition (POD) data approximation following \cite{singler2014new,chapelle2012galerkin}.  Since we derive the HDG-POD reduced order model in continuous time, we consider a continuous time framework here; time discrete data can be handled similarly.  Let data $u\in L^2(0,T; X)$ be given, where  $X$ is a Hilbert space with inner product $(\cdot,\cdot)$ and norm $\|\cdot\|$. The POD data approximation problem for the given data $u\in  L^2(0,T; X)$ looks for an orthonormal basis $\{\varphi_i\}\subset X$ (the POD modes) minimizing the data approximation error 
\begin{align*}
	E_r = \min_{\Pi}\norm{u - \Pi u}_{L^2(0,T;X)}^2
\end{align*}
for the data approximation 
\begin{align*}
	\Pi u = \sum_{i=1}^r (u(t),\varphi_i) \varphi_i.
\end{align*}

To solve this data approximation problem, introduce the linear operator $K: L^2(0,T)\to X$ defined by 
\begin{align*}
	K w = \int_0^T u(t) w(t) dt.
\end{align*}
It can be shown that $K$ is a compact linear operator, and therefore it has a singular value decomposition (SVD): there exist singular values $\{\sigma_i\}$ and singular vectors $\{f_i\}\subset L^2(0,T)$ and $\{\varphi_i\}\subset X$ such that 
\begin{align*}
	K f_i = \sigma_i \varphi_i,\qquad K^* \varphi_i = \sigma_i f_i,
\end{align*}
where the Hilbert adjoint operator $K^*:X\to L^2(0,T)$ is given by
\begin{align*}
	[K^*x](t) = (x, u(t)).
\end{align*}
We assume throughout that the singular values are ordered so that $\sigma_1\ge\sigma_2\ge\cdots\ge 0$.  The squares of the singular values, $\lambda_i = \sigma_i^2$, are called the POD eigenvalues of the data.  It can be shown that the POD modes $\{\varphi_i\}\subset X$ are the above singular vectors of $K$.

Define the linear operator $\mathcal R = K{K}^*:X\to X$ by
\begin{align*}
	\mathcal R x= \int_0^T (x,u(t))u(t) dt.
\end{align*}
The nonzero eigenvalues of $\mathcal R$ are the nonzero POD eigenvalues and the corresponding eigenvectors of $\mathcal R$ are the POD modes, i.e., the singular vectors $\{\varphi_i\}\subset X$.  Moreover, if $X$ is finite dimensional, then there are only finitely many nonzero POD singular values.

The singular vectors are orthonormal bases for each space. Furthermore, $K$ is known to be a Hilbert-Schmidt operator, and so the sum of the squares of the singular values is finite:
\begin{align*}
	\sum_{i=1}^{\infty} \sigma_i^2 < \infty.
\end{align*}
The POD data approximation error is given by the sum of the squares of the neglected singular values, i.e.,
\begin{align}\label{POD_pro_error}
	E_r = \min_{\Pi}\norm{u - \Pi u}_{L^2(0,T;X)}^2 = \sum_{i>r} \sigma_i^2 = \sum_{i>r} \lambda_i.
\end{align} 

\subsection{HDG-POD Reduced Order Model}
\label{sec:HDG-POD_details}

Next, we derive the HDG-POD reduced order model (ROM).  We first consider POD data approximation to construct low order function spaces, then we give a basic formulation of the ROM, and then describe the implementation.

\subsubsection{Data Approximation and POD Modes}

The HDG method produces the flux $\bm q_h(t) $, the scalar variable $u_h(t)$, and the numerical trace $\widehat u_h(t)$.  We consider this data as either
\begin{enumerate}
	\item  the exact solution data for the continuous time HDG formulation \eqref{HDG_discrete2_a}-\eqref{HDG_discrete2_e}, or
	\item  a continuous time extension of the exact solution data for a time discretized HDG formulation.
\end{enumerate}
The first case is useful for the error analysis of the continuous time HDG-POD ROM that we perform in \Cref{sec:error_analysis}.  The second case is useful for the actual computation.  We note there are many ways to extend time discrete data to continuous time data.  The simplest approach is to extend to a piecewise constant function in time.  

For the data above, define the following three different linear POD operators
\begin{align*}
	K^{\bm q}: L^2(0,T)\to [L^2(\Omega)]^d, \quad K^{\bm q}  w_1 &= \int_0^T \bm q_h(t) w_1(t) dt,\\
	K^{u}: L^2(0,T)\to L^2(\Omega), \quad K^u  w_2 &= \int_0^T u_h(t)w_2(t) dt,\\
	K^{\widehat u}: L^2(0,T)\to L^2(\varepsilon_h^o), \quad K^{\widehat u}  w_3 &= \int_0^T \widehat u_h(t)w_3(t) dt.
\end{align*}
The POD data approximation theory applies for each operator.  Since $ \bm q_h $, $u_h$ and $\widehat{u}_h$ are each expressed in terms of a finite basis, these operators each have finite rank and therefore only have finitely many nonzero singular values.

We let $\{\sigma_j^{\bm q}, f_j^{\bm q},\varphi_j^{\bm q}\}$, $\{\sigma_j^{u}, f_j^{u},\varphi_j^{u}\}$, and $\{\sigma_j^{\widehat u}, f_j^{\widehat u},\varphi_j^{\widehat u}\}$ denote the singular values and singular vectors of $K^{\bm q}$, $K^{u}$, and $K^{\widehat u}$, respectively. Since the POD modes $\varphi_j^{\bm q}, \varphi_j^{u}$, and $\varphi_j^{\widehat u}$ are in  $\bm V_h $, $ W_h$, and $M_h$, respectively, there exists coefficients $\left[D^{i,j}_1\right]_{i,j=1}^{N_1,r_1}$, $\left[D^{i,j}_2\right]_{i,j=1}^{N_2,r_2}$, and $\left[D^{i,j}_3\right]_{i,j=1}^{N_3,r_3}$ such that
\begin{equation}\label{Heat_POD_modes}
\begin{split}
\varphi_j^{\bm q} = \sum_{i=1}^{N_1} D_1^{i,j} \bm \varphi_i,\quad j=1,2,\ldots,r_1,\\
\varphi_j^{u} = \sum_{i=1}^{N_2} D_2^{i,j}  \phi_i,\quad j=1,2,\ldots,r_2,\\
\varphi_j^{\widehat u} = \sum_{i=1}^{N_3} D_3^{i,j}  \psi_i,\quad j=1,2,\ldots,r_3.
\end{split}
\end{equation}

Define the spaces of POD modes for $\bm q_h$, $u_h$ and $\widehat u_h$ by $\bm V_h^{r_1} = \text{span} \{\varphi_i^{\bm q}\}_{i=1}^{r_1}$,  $ W_h^{r_2} = \text{span} \{\varphi_i^{u}\}_{i=1}^{r_2}$, and  $ M_h^{r_3} = \text{span} \{\varphi_i^{\widehat u}\}_{i=1}^{r_3}$, respectively.

Computing the three necessary singular value decompositions can be done efficiently using incremental approaches; see, e.g., \cite{Brand06,BakerGallivan12,IwenOng16,Oxberry17,FareedShenSinglerZhang17} and the references therein.

\subsubsection{HDG-POD Reduced Order Model: Basic Formulation}
Next, we present the HDG-POD reduced order model. In \Cref{sec:HDG_local_solver}, we saw the local solver is one of the main advantages of HDG methods since it makes the global degrees of freedom significant smaller compared to other DG methods. However, this creates the main difficulty to devising an HDG-POD reduced model since we don't have an ordinary dynamical system for $\widehat u_h$ in the HDG method.

To derive the HDG-POD reduced order model, we perform the following procedure.  First, we consider the continuous time HDG formulation \eqref{Heat_HDG} without any elimination of variables, and project this system onto the POD spaces $\bm V_h^{r_1} $, $ W_h^{r_2} $, and $ M_h^{r_3}$ to obtain an initial reduced order model.  Next, we eliminate the reduced flux and numerical trace in the reduced order model to further reduce the order.  Finally, if desired, we recover the flux using a simple relationship with $\mathcal O(r)$ computational cost, where $r$ is the order of reduced model.  We present details of the above procedure below, and also show the HDG-POD ROM is well-posed.

We emphasize that the global variable for HDG-POD is the scalar variable, not the numerical trace; this is totally different from HDG methods. 

Projecting the continuous time HDG system (\ref{Heat_HDG}) onto the POD spaces $\bm V_h^{r_1}\times W_h^{r_2}\times M_h^{r_3}$ gives the initial HDG-POD reduced order model.  Specifically, we seek $(\bm q_r,u_r,\widehat u_r) \in \bm V_h^{r_1}\times W_h^{r_2}\times M_h^{r_3}$ satisfying
\begin{subequations}\label{Heat}
	\begin{align}
		(c\bm{q}_r,\bm{v})_{\mathcal{T}_h}-(u_r,\nabla\cdot \bm{v})_{\mathcal{T}_h}+\left\langle\widehat{u}_r,\bm{v\cdot n} \right\rangle_{\partial {\mathcal{T}_h}} &= 0 \label{Heat_a}\\
		(\partial_tu_r,w)_{\mathcal T_h}-(\bm{q}_r,\nabla w)_{\mathcal{T}_h}+\left\langle\bm q_r\cdot \bm n +\tau(u_r-\widehat u_r),w\right\rangle_{\partial {\mathcal{T}_h}} &= (f,w)_{\mathcal{T}_h}\label{Heat_b}\\
		\left\langle \bm q_r\cdot \bm n +\tau(u_r-\widehat u_r), \mu\right\rangle_{\partial{\mathcal{T}_h}\backslash\varepsilon_h^{\partial}} &=0,\label{Heat_c}\\
		u_r(\cdot,0) &= u_{r,0},
	\end{align}
\end{subequations}
for all $\bm{v}$ in $\bm{V}_h^{r_1}$, $w\in W_h^{r_2}$, and $\mu\in M_h^{r_3}$, and $u_{r,0}\in W_h^{r_2}$ is the initial condition. Integration by parts in \eqref{Heat_b} gives
\begin{subequations}\label{HDG_POD}
	\begin{align}
		(c\bm{q}_r,\bm{v})_{\mathcal{T}_h}-(u_r,\nabla\cdot \bm{v})_{\mathcal{T}_h}+\left\langle\widehat{u}_r,\bm{v\cdot n} \right\rangle_{\partial {\mathcal{T}_h}}&= 0, \label{HDG_POD_a}\\
		(\partial_tu_r,w)_{\mathcal T_h}+(\nabla \cdot \bm{q}_r, w)_{\mathcal{T}_h}+\left\langle\tau(u_r-\widehat u_r),w\right\rangle_{\partial {\mathcal{T}_h}} &= (f,w)_{\mathcal{T}_h},\label{HDG_POD_b}\\
		\left\langle \bm q_r\cdot \bm n +\tau(u_r-\widehat u_r), \mu\right\rangle_{\partial{\mathcal{T}_h}\backslash\varepsilon_h^{\partial}} &=0,\label{HDG_POD_c}\\
		u_r(\cdot,0) &=u_{r,0}.
	\end{align}
\end{subequations}

The system \eqref{HDG_POD} is a reduced model but not our final reduced order model since it has three variables $\bm q_r $, $u_r$, and $\widehat u_r$.  We utilize \eqref{HDG_POD_a} and \eqref{HDG_POD_c} to eliminate the reduced flux and numerical trace as follows: let $(\bm q_r^{u_r}, \widehat u_r^{u_r}) \in \bm{V}_h^{r_1} \times M_h^{r_3} $ be the unique solution of 
\begin{subequations}\label{reduced_tec}
	\begin{align}
		(c\bm{q}_r^{u_r},\bm{v})_{\mathcal{T}_h}+\left\langle\widehat{u}_r^{u_r},\bm{v\cdot n} \right\rangle_{\partial {\mathcal{T}_h}}&= (u_r,\nabla\cdot \bm{v})_{\mathcal{T}_h},\label{local_eli_a}\\
		\left\langle -\bm q_r^{u_r}\cdot \bm n + \tau\widehat u_r^{u_r}, \mu\right\rangle_{\partial{\mathcal{T}_h}\backslash \varepsilon_h^\partial} &=\left\langle \tau u_r,  \mu\right\rangle_{\partial{\mathcal{T}_h}\backslash \varepsilon_h^\partial},\label{local_eli_b}
	\end{align}
\end{subequations}
for all $(\bm v, \mu) \in \bm{V}_h^{r_1} \times M_h^{r_3}$.  The system \eqref{reduced_tec} is uniquely solvable since the bilinear form on the left hand side is positive definite on $\bm{V}_h^{r_1} \times M_h^{r_3}$. In general, for $ w \in W_h^{r_2} $ we define $(\bm q_r^w, \widehat u_r^w)$ to be the unique solution of 
\begin{subequations}
	\begin{align}
		(c\bm{q}_r^w,\bm{v})_{\mathcal{T}_h}+\left\langle\widehat{u}_r^w,\bm{v\cdot n} \right\rangle_{\partial {\mathcal{T}_h}}&= (w,\nabla\cdot \bm{v})_{\mathcal{T}_h},\label{local_eli_c}\\
		\left\langle -\bm q_r^w\cdot \bm n +\tau\widehat u_r^w, \mu\right\rangle_{\partial{\mathcal{T}_h}\backslash \varepsilon_h^\partial} &=\left\langle \tau w,  \mu\right\rangle_{\partial{\mathcal{T}_h}\backslash \varepsilon_h^\partial}.\label{local_eli_d}
	\end{align}
\end{subequations}

This allows us to rewrite the reduced order model.
\begin{theorem}
	The HDG-POD reduced order model \eqref{HDG_POD} can be rewritten as follows: find $ u_r \in W_h^{r_2} $ satisfying
	\begin{subequations}\label{HDG-POD_ODE}
		\begin{align}
			(\partial_tu_r,w)_{\mathcal T_h}+a_r(u_r,w) &= \ell_r (w),\\
			u_r(\cdot,0) &=u_{r,0},
		\end{align}
	\end{subequations}
	for all $ w \in W_h^{r_2} $.  Here, the forms are given by
	\begin{align}
		a_r(u_r,w) & = (c\bm q_r^{u_r},\bm q_r^w)_{\mathcal T_h} +\left\langle \tau (u_r-\widehat u_r^{u_r}), w-\widehat u_r^{w}\right\rangle_{\partial \mathcal T_h},\label{bilin_a}\\
		\ell_r (w) & = (f,w)_{\mathcal T_h},\label{bilin_b}
	\end{align}
	for $u_r,w \in W_h^{r_2}$.  Also, the bilinear form $a_r$ defined on $W_h^{r_2}$ is symmetric and semipositive definite.
\end{theorem}
\begin{proof}
	First, from \eqref{HDG_POD_b} we have
	\begin{align*}
		(f,w)_{\mathcal{T}_h} =	(\partial_tu_r,w)_{\mathcal T_h}+(\nabla \cdot \bm{q}_r, w)_{\mathcal{T}_h}+\left\langle\tau(u_r-\widehat u_r),w\right\rangle_{\partial {\mathcal{T}_h}}.
	\end{align*}
	Here $\bm q_r $ and $\widehat u_r$ is the solution of \eqref{reduced_tec}, then we have 
	\begin{align*}
		(f,w)_{\mathcal{T}_h} = (\partial_tu_r,w)_{\mathcal T_h}+(\nabla \cdot \bm{q}_r^{u_r}, w)_{\mathcal{T}_h}+\left\langle\tau(u_r-\widehat u_r^{u_r}),w\right\rangle_{\partial {\mathcal{T}_h}}.
	\end{align*}
	Choose  $\bm v = \bm{q}_r^{u_r}$ in \eqref{local_eli_c} to obtain
	\begin{align*}
		(c\bm{q}_r^w,\bm{q}_r^{u_r})_{\mathcal{T}_h}+\left\langle\widehat{u}_r^w,\bm{q}_r^{u_r}\cdot \bm n \right\rangle_{\partial {\mathcal{T}_h}}= (w,\nabla\cdot \bm{q}_r^{u_r})_{\mathcal{T}_h}.
	\end{align*}
	Similarly, choose $\mu = \widehat{u}_r^w$ in  \eqref{local_eli_b} and recall that $\mu =0$ on $\varepsilon_h^\partial$ to get
	\begin{align*}
		\left\langle \bm q_r^{u_r}\cdot \bm n , \widehat{u}_r^w\right\rangle_{\partial{\mathcal{T}_h}} =\left\langle \tau(\widehat u_r^{u_r} - u_r),  \widehat{u}_r^w\right\rangle_{\partial{\mathcal{T}_h}}.
	\end{align*}
	So, we have
	\begin{align*}
		\hspace{1em}&\hspace{-1em}(f,w)_{\mathcal{T}_h}\\
		 &=	(\partial_tu_r,w)_{\mathcal T_h}+(\nabla \cdot \bm{q}_r, w)_{\mathcal{T}_h}+\left\langle\tau(u_r-\widehat u_r),w\right\rangle_{\partial {\mathcal{T}_h}} \\
		& = (\partial_tu_r,w)_{\mathcal T_h}+(\nabla \cdot \bm{q}_r^{u_r}, w)_{\mathcal{T}_h}+\left\langle\tau(u_r-\widehat u_r^{u_r}),w\right\rangle_{\partial {\mathcal{T}_h}} \\
		&= (\partial_tu_r,w)_{\mathcal T_h} + (c\bm{q}_r^w,\bm{q}_r^{u_r})_{\mathcal{T}_h}+\left\langle\widehat{u}_r^w,{\bm{q}_r^{u_r}} \cdot \bm n \right\rangle_{\partial {\mathcal{T}_h}}+\left\langle\tau(u_r-\widehat u_r^{u_r}),w\right\rangle_{\partial {\mathcal{T}_h}} \\
		&= (\partial_tu_r,w)_{\mathcal T_h} + (c\bm{q}_r^w,\bm{q}_r^{u_r})_{\mathcal{T}_h}-\left\langle  \tau(u_r -\widehat u_r^{u_r}),\widehat{u}_r^w \right\rangle_{\partial{\mathcal{T}_h}} +\left\langle\tau(u_r-\widehat u_r^{u_r}),w\right\rangle_{\partial {\mathcal{T}_h}} \\
		&= (\partial_tu_r,w)_{\mathcal T_h} + (c\bm{q}_r^w,\bm{q}_r^{u_r})_{\mathcal{T}_h} +\left\langle\tau(u_r-\widehat u_r^{u_r}),w - \widehat u_r^{w}\right\rangle_{\partial{\mathcal{T}_h}}.
	\end{align*}
	
	Next, it is clear that $ a_r $ is symmetric and semipositive definite on $ W_h^{r_2} $.
\end{proof}

\begin{remark}
	Equation \eqref{HDG-POD_ODE} is our final HDG-POD reduced model; the only variable is the reduced scalar variable.
\end{remark}

\subsubsection{HDG-POD Implementation}

Assume 
\begin{align}\label{HDG_POD_expre}
	&\bm q_{r}=  \sum_{j=1}^{r_1} a_j(t)\varphi_{j}^{\bm q},  \quad 
	u_r = \sum_{j=1}^{r_2}b_{j}(t)\varphi_{j}^{u},  \quad 
	\widehat{u}_r = \sum_{j=1}^{r_3}c_{j}(t)\varphi_{j}^{\widehat u}.
\end{align}
Substitute \eqref{HDG_POD_expre} into the HDG-POD reduced order model with all variables \eqref{HDG_POD_a}-\eqref{HDG_POD_c} and use the corresponding  test functions to test  \eqref{HDG_POD_a}-\eqref{HDG_POD_c}, respectively, to obtain the following finite dimensional dynamical system:
\begin{align}\label{Matrix_HDG_POD}
	\bm M_r \bm \dot{\bm x}_r +\bm A_r \bm x_r &= \bm b_r, \\
	\bm x_r(0)& = \bm x_0,
\end{align}
where
\begin{align*}
	\bm x_r(t) &=[a_1(t),\ldots,a_{r_1}(t),b_1(t),\ldots,b_{r_2}(t),c_1(t),\ldots,c_{r_3}(t)]^T,\\
	\bm x_0 &= [(u_{r,0},\varphi_1^u)_{\mathcal T_h},(u_{r,0},\varphi_2^u)_{\mathcal T_h},\ldots,(u_{r,0},\varphi_{r_2}^u)_{\mathcal T_h}]^T,
\end{align*}
and
\begin{equation}\label{system_equation_HDG_POD}
\begin{split}
\bm M_r &=\begin{bmatrix}
D_1 &0&0\\
0 & D_2&0\\
0 &0&D_3
\end{bmatrix}^{T}
\begin{bmatrix}
0 &0&0\\
0 &M&0\\
0 &0&0
\end{bmatrix}
\begin{bmatrix}
D_1 &0&0\\
0 & D_2&0\\
0 &0&D_3
\end{bmatrix}
=  \begin{bmatrix}
0 &0&0\\
0 &{I}_{r_2}&0\\
0 &0&0
\end{bmatrix},\\
\bm A_r &=\begin{bmatrix}
D_1 &0&0\\
0 & D_2&0\\
0 &0& D_3
\end{bmatrix}^{T}
\begin{bmatrix}
A_1 &-A_2&A_3\\
A_2^T&A_4&-A_5\\
A_3^T &A_5^T&-A_6
\end{bmatrix}
\begin{bmatrix}
D_1 &0&0\\
0 & D_2&0\\
0 &0& D_3
\end{bmatrix}
\\
&=  \begin{bmatrix}
B_1 &-B_2& B_3\\
B_2^T & B_4&- B_5\\
B_3^T & B_5^T&-B_6
\end{bmatrix},\\
\bm b_r &=
\begin{bmatrix}
D_1 &0&0\\
0 & D_2&0\\
0 &0& D_3
\end{bmatrix}^{T}
\left[ {\begin{array}{*{20}{c}}
	0\\
	b(t)\\
	0\\
	\end{array}} \right]
=\left[ {\begin{array}{*{20}{c}}
	0\\
	D_2^T b(t)\\
	0\\
	\end{array}} \right]
=\left[ {\begin{array}{*{20}{c}}
	0\\
	z(t)\\
	0\\
	\end{array}} \right].
\end{split}
\end{equation}
Here, $z$ and $\{B_j\}_{j=1}^6$ are determined by \eqref{system_equation_HDG_POD}, and $I_{r}$ denotes the identity matrix of order $r$.

Since $\{\varphi_j^{\bm q}\}_{j=1}^{r_1} $, $\{\varphi_i^{u}\}_{j=1}^{r_2}$, and $\{\varphi_i^{\widehat u}\}_{j=1}^{r_3}$ are orthonormal bases in $\bm V_h^{r_1} $, $ W_h^{r_2}$, and $M_h^{r_3}$ respectively, from equation (\ref{Heat_POD_modes}) we have 
\begin{equation}\label{ded}
\begin{split}
I_{r_1} &= \left[(\varphi_j^{\bm q},\varphi_i^{\bm q} )_{\mathcal T_h}\right]_{i,j=1}^{r_1} = \left[\left(\sum_{\ell=1}^{N_1} D_1^{\ell,j} \bm \varphi_{\ell},\sum_{k=1}^{N_1} D_1^{k,j}\bm \varphi_k \right)_{\mathcal T_h}\right]_{i,k=1}^{r_1} = D_1^TA_7D_1,\\
I_{r_2} &= \left[(\varphi_j^{u},\varphi_i^{u} )_{\mathcal T_h}\right]_{i,j=1}^{r_2} = \left[\left(\sum_{\ell=1}^{N_2} D_2^{\ell,j}  \phi_{\ell},\sum_{k=1}^{N_2} D_2^{k,j} \phi_k \right)_{\mathcal T_h}\right]_{i,k=1}^{r_2} = D_2^TMD_2,\\
I_{r_3} &= \left[\left\langle\varphi_j^{\widehat u},\varphi_i^{\widehat u}\right\rangle_{\partial\mathcal T_h}\right]_{i,j=1}^{r_3} = \left[\left\langle\sum_{\ell=1}^{N_3} D_3^{\ell,j}  \psi_{\ell},\sum_{k=1}^{N_3} D_3^{k,j} \psi_k \right\rangle_{\partial \mathcal T_h}\right]_{i,k=1}^{r_3} = D_3^T A_8 D_3,\\
\end{split}
\end{equation}
where $A_7= [(\bm\varphi_j,\bm\varphi_i )_{\mathcal{T}_h}] $ and $ A_8 = [\left\langle \psi_j,{\psi_i}\right\rangle_{\partial\mathcal{T}_h}]$.  To obtain a reduced order model in terms of the reduced scalar coefficients below, we show $ B_1 $ and $ B_6 $ are positive definite and therefore invertible:  For $\tau_* = \min_{K\in\partial \mathcal T_h}{\tau}>0$ and $c>c_0$, we have 
\begin{align*}
	B_1 = D_1^TA_1D_1 \ge c_0 D_1^TA_7D_1 = c_0 I_{r_1}>0,\\
	B_6 = D_3^TA_6D_3 \ge \tau_* D_3^TA_8D_3 = \tau_* I_{r_3}>0.
\end{align*}


Let $\bm a = [a_1(t),\ldots,a_{r_1}(t)]^T $, $\bm b = [b_1(t),\ldots,b_{r_2}(t)]^T$, and $\bm c =[c_1(t),\ldots,c_{r_3}(t)]$.  The system \eqref{Matrix_HDG_POD} can be rewritten as
\begin{subequations}\label{algebra_HDG_POD}
	\begin{align}
		B_1 \bm a- B_2\bm b + B_3\bm c &= 0,\label{algebra_HDG_POD_1}\\
		\dot{\bm b} +  B_2^T\bm a+ B_4\bm b - B_5\bm c&= z,\label{algebra_HDG_POD_2}\\
		B_3^T\bm a+  B_5^T\bm b -B_6\bm c &= 0.\label{algebra_HDG_POD_3}
	\end{align}
\end{subequations}
Since $ B_1 $ and $ B_6 $ are invertible, we can solve \eqref{algebra_HDG_POD_1} and \eqref{algebra_HDG_POD_3}.  We obtain
\begin{subequations}
	\begin{align}
		\bm a = \left(B_1^{-1}B_2 - B_1^{-1}B_3(B_6+B_3^TB_1^{-1}B_3)^{-1}(B_5^T+B_3^TB_1^{-1}B_2)\right)\bm{b} &=: G\bm{b},\label{algebra_HDG_POD_4}\\
		\bm c = \left(B_6+B_3^TB_1^{-1}B_3\right)^{-1} (B_5^T + B_3^T B_1^{-1}B_2) \bm{b} &=: H\bm b.\label{algebra_HDG_POD_5}
	\end{align}
\end{subequations}
Substitute \eqref{algebra_HDG_POD_4} and \eqref{algebra_HDG_POD_5} into \eqref{algebra_HDG_POD_2}  to obtain the system for $\bm b$ only:
\begin{align}\label{algebra_HDG_POD_6}
	\dot{\bm b} +  \left( B_2^T G +B_4 - B_5H\right) \bm{b}= z.
\end{align}
\begin{remark}
	Equation \eqref{algebra_HDG_POD_6} is the matrix form of the final reduced order model \eqref{HDG-POD_ODE} written in terms of the scalar unknown $ u_r $ only.  Also, in \eqref{algebra_HDG_POD_4}  and \eqref{algebra_HDG_POD_5} we can invert the matrix $B_6+B_3^TB_1^{-1}B_3$ since $B_1$ and $B_6$ are positive definite; furthermore, this is computationally cheap since we only compute this inverse once and the dimension of the matrix is low.  Finally, after we solve \eqref{algebra_HDG_POD_6}, we can recover the reduced flux coefficients $\bm a$ from  \eqref{algebra_HDG_POD_4} with computational cost $\mathcal O(r_1)$.
\end{remark}

It is obvious the system \eqref{algebra_HDG_POD_6} has a unique solution $\bm{b}$, and also $\bm a$ and $\bm c$ are uniquely determined by  \eqref{algebra_HDG_POD_4} and \eqref{algebra_HDG_POD_5}.  This proves the well-posedness of the reduced order model \eqref{HDG-POD_ODE} for $ u_r $ given earlier:
\begin{theorem}
	For any $ u_{r,0} \in W_h^{r_2} $ and any $ T > 0 $, there exists a unique solution $ u_r $ of the reduced order model \eqref{HDG-POD_ODE} on the time interval $ (0,T) $.
\end{theorem}

\section{Error analysis}
\label{sec:error_analysis}

Next, we perform an error analysis of the HDG-POD reduced order model.  We assume the solution data of the continuous time HDG formulation \eqref{HDG_discrete2} is used to generate the POD spaces and the HDG-POD reduced order model, and we bound the error between the solution of the continuous time HDG-POD reduced model and the continuous time HDG formulation.  As mentioned in the introduction, this type of error analysis has been performed for many standard and modified POD reduced order models of many different PDEs.  For more discussion about this type of analysis, see, e.g., \cite[Section 3.6]{kunisch2001galerkin} and \cite[Section 1]{kunisch2002galerkin}.

In this first work on HDG-POD model order reduction, we focus on the convergence analysis in the continuous time case to emphasize the errors caused by the HDG-POD spatial discretization only.

We begin the analysis by discussing POD projection operators and their error estimates in \Cref{sec:POD_projection_errors}.  Then we use energy arguments in \Cref{sec:energy_argument} to obtain the error bounds. 

\subsection{POD Projection Errors}
\label{sec:POD_projection_errors}

First, we introduce the broken Sobolev spaces $ H^1(\mathcal{T}_h) $ and $ H(\text{div};\mathcal{T}_h) $:
\begin{align*}
	H^1(\mathcal T_h) &= \{v\in L^2(\Omega): \forall K\in \mathcal T_h, v|_K \in H^1(K)\},\\
	H(\text{div};\mathcal T_h) &= \{ \bm v \in [L^2(\Omega)]^d: \forall K\in \mathcal T_h, \nabla \cdot \bm v |_K \in L^2(K)\}.
\end{align*}
On these spaces, we have the broken gradient and divergence seminorms, respectively:
\begin{align*}
	\norm{\nabla v}_{\mathcal T_h} = \left( \sum_{K\in \mathcal T_h} \norm{\nabla v}_{L^2(K)}^2\right)^{1/2},  \quad  \norm{\nabla \cdot \bm v}_{\mathcal T_h} = \left( \sum_{K\in \mathcal T_h} \norm{\nabla \cdot \bm v}_{L^2(K)}^2\right)^{1/2}.
\end{align*}

The three linear POD operators $K^{\bm q}$, $K^{u}$, and $K^{\widehat u}$ define the following three orthogonal POD projection operators $\bm \Pi: \bm V_h\to \bm V_h^{r_1}$, $ \Pi: W_h\to  W_h^{r_2}$ and $P_M: M_h\to M_h^{r_3}$ that satisfy
\begin{equation}\label{projection}
\begin{split}
(\bm \Pi \bm v_h,\varphi_i^{\bm q})_{\mathcal T_h} &=  (\bm v_h,\varphi_i^{\bm q})_{\mathcal T_h},\; i=1,2\ldots,r_1,\\ 
(\Pi  w_h,\varphi_i^{u})_{\mathcal T_h} &=  (w_h,\varphi_i^{u})_{\mathcal T_h},\; i=1,2\ldots,r_2,\\ 
\left\langle  P_M \mu_h,\varphi_i^{\widehat u}\right\rangle_{\partial\mathcal T_h} &= \left\langle  \mu_h,\varphi_i^{\widehat u}\right\rangle_{\partial\mathcal T_h},\; i=1,2\ldots,r_3,\\
\end{split}
\end{equation}
for all $ (\bm v_h, w_h, \mu_h ) \in ( \bm V_h, W_h, M_h ) $.  Furthermore, we have the following projection errors
\begin{equation}\label{projection_error}
\begin{split}
&\int_0^T \norm{\bm q_h - \bm\Pi \bm q_h}_{\mathcal T_h}^2 dt = \sum_{i>r_1} {\lambda_i^{\bm q}}<\infty,\\
&\int_0^T \norm{u_h - \Pi u_h}_{\mathcal T_h}^2 = \sum_{i>r_2}{\lambda_i^{u}}<\infty,\\
&\int_0^T \norm{\widehat u_h - P_M \widehat u_h}_{\partial \mathcal T_h}^2 dt= \sum_{i>r_3}{\lambda_i^{\widehat u}}<\infty.
\end{split}
\end{equation}

In the POD projection errors in \eqref{projection_error} above, $ L^2 $ norms are used in space.  Below, we give the POD projection error formulas in various seminorms.  POD projection error formulas using a different norm was first established for continuous time data in \cite{singler2014new}, and then for time discrete data in \cite{iliescu2014variational}.  These POD projection error formulas allow us to avoid using inverse inequalities with the POD spaces; it is well known that the constants in these inequalities blow up quickly as $ r $ increases.  Using the following result enables us to prove that the HDG-POD error bounds converge to zero as $ r $ increases.  See \cite{singler2014new,iliescu2014variational} for more information.
\begin{lemma}\label{POD_error_H1}
	The POD projection errors for $u_h$ and $ \bm q_h $ satisfy
	\begin{align*}
		\int_0^T \| \nabla (u_h - \Pi u_h) \|_{\mathcal T_h}^2 dt &= \sum_{i>r_2} \lambda_i^u \|\nabla \varphi_i^u\|_{\mathcal T_h}^2,\\
		\int_0^T \| u_h - \Pi u_h \|_{\partial \mathcal T_h}^2 dt &= \sum_{i>r_2} \lambda_i^u \|\varphi_i^u\|_{\partial \mathcal T_h}^2,\\
		\int_0^T \| \nabla \cdot (\bm q_h - \bm \Pi \bm q_h ) \|_{\mathcal T_h}^2 dt &= \sum_{i>r_1} \lambda_i^{\bm q} \| \nabla \cdot \varphi_i^{\bm q} \|_{\mathcal T_h}^2,\\
		\int_0^T \| ( \bm q_h - \bm \Pi \bm q_h ) \cdot \bm n \|_{\partial \mathcal T_h}^2 dt &= \sum_{i>r_1} \lambda_i^{\bm q} \|\varphi_i^{\bm q} \cdot \bm n\|_{\partial \mathcal T_h}^2.
	\end{align*}
\end{lemma}
\begin{proof}
	We prove the first identity; the proofs of the remaining identities are similar.
	
	Since the space $ W_h $ is finite dimensional, the number of nonzero POD singular values for $ u_h $ is finite.  Also, each POD mode corresponding to a nonzero singular value is in $ W_h $.  The POD projection error formula \ref{projection_error} gives
	$$
	  u_h = \sum_{i=1}^{s_u} (u_h,\varphi_i^u)_{\mathcal T_h} \varphi_i^u,
	$$
	where $ s_u $ is the number of nonzero POD singular values for $ u_h $.  Therefore,
	$$
	  u_h - \Pi u_h = \sum_{i=r_2+1}^{s_u} (u_h,\varphi_i^u)_{\mathcal T_h} \varphi_i^u.
	$$
	The proof of the error formula is now similar to the time discrete data case in \cite{iliescu2014variational}.  We have
	\begin{align*}
		\int_0^T \| \nabla ( u_h - \Pi u_h ) \|_{\mathcal T_h}^2 dt &= \int_0^T \sum_{K\in \mathcal T_h}\norm {\sum_{i>r_2} (u_h,\varphi_i^u)_{\mathcal T_h} \nabla\varphi_i^u}_{L^2(K)}^2 dt\\
		&= \int_0^T \left( \sum_{i>r_2}(u_h,\varphi_i^u)_{\mathcal T_h}\nabla \varphi_i^u, \sum_{j>r_2}(u_h,\varphi_j^u)_{\mathcal T_h}\nabla\varphi_j^u \right)_{\mathcal T_h}dt\\
		&= \int_0^T \sum_{i>r_2}  \sum_{j>r_2}  (u_h,\varphi_i^u)_{\mathcal T_h}(u_h,\varphi_j^u)_{\mathcal T_h}(\nabla\varphi_i^u, \nabla\varphi_j^u)_{\mathcal T_h}dt\\
		&= \sum_{i>r_2}  \sum_{j>r_2} (\nabla\varphi_i^u, \nabla\varphi_j^u)_{\mathcal T_h}\left(\int_0^T (u_h,\varphi_j^u)_{\mathcal T_h}u_h dt,\varphi_i^u\right)_{\mathcal T_h}\\
		&= \sum_{i>r_2}  \sum_{j>r_2} (\nabla\varphi_i^u, \nabla\varphi_j^u)_{\mathcal T_h}(\mathcal R^u \varphi_j^u,\varphi_i^u)_{\mathcal T_h}\\
		&= \sum_{i>r_2} \lambda_i^u \| \nabla \varphi_i^u \|_{\mathcal T_h}^2,
	\end{align*}
	where we have used $\mathcal R^u \varphi_j^u = \lambda_j^u \varphi_j^u$ and the fact that the POD modes are orthonormal in the inner product $(\cdot,\cdot)_{\mathcal T_h}$.
\end{proof}

\subsection{Error Analysis by Energy Estimates}
\label{sec:energy_argument}

Next, we compare the solution of the HDG-POD reduced order model to the POD projections of the HDG solution data.  Some of our analysis relies on techniques used to analyze the continuous time HDG formulation for the heat equation in \cite{chabaud2012uniform}.  However, we also use different methods due to the POD projections.

We begin by recalling a trace inequality that can be found in \cite{riviere2008discontinuous}.
\begin{lemma}\label{trace_ine}
	Let $v\in \mathcal P^k(K)$.  There exists $C>0$, independent of $K\in\mathcal T_h$ and $v$, but dependent on the polynomial degree $k$, such that 
	\begin{align}\label{tra}
		\norm{v}_{\partial K} \le C h_K^{-1/2} \norm{v}_K,
	\end{align}
	where $h_K$ is the diameter of $K$.
\end{lemma}
\begin{remark}
	In the case where $K\in\mathcal T_h$ is a triangle in 2D, or a tetrahedron in 3D, exact expressions are known for the constant $C$ in \Cref{trace_ine} as a function of the polynomial degree $k$:
	\begin{align*}
		C = \sqrt{\frac{(k+1)(k+2)}{2}} \;\text{in} \;2D \quad \text{and} \quad C = \sqrt{\frac{(k+1)(k+3)}{3}} \;\text{in} \;3D.
	\end{align*}
\end{remark}

We split the errors using the POD projections, and begin by obtaining equations for the components of the errors residing in the POD subspaces.
\begin{lemma}
	Let $\varepsilon_r^{\bm q}  = \bm\Pi \bm q_h - \bm q_r$, $\varepsilon_r^{u}  = \Pi u_h - u_r$,  and $\varepsilon_r^{\widehat u}  = P_M u_h - \widehat u_r$. Then
	\begin{subequations}\label{POD_error}
		\begin{align}
			\hspace{2em}&\hspace{-2em}(c\varepsilon_r^{\bm q} ,\bm{v})_{\mathcal{T}_h}-(\varepsilon_r^u,\nabla\cdot \bm{v})_{\mathcal{T}_h}+\left\langle\varepsilon_r^{\widehat{u}},\bm{v\cdot n} \right\rangle_{\partial {\mathcal{T}_h}}\nonumber\\
			&= (c\bm\Pi \bm{q}_h -c \bm{q}_h ,\bm{v})_{\mathcal{T}_h}
			+(u_h - \Pi u_h,\nabla\cdot \bm{v})_{\mathcal{T}_h} + \left\langle P_M \widehat{u}_h -\widehat  u_h,\bm{v\cdot n} \right\rangle_{\partial {\mathcal{T}_h}}, \label{POD_error_a}\\
			\hspace{2em}&\hspace{-2em} (\partial_t\varepsilon^u_r,w)_{\mathcal T_h}-(\varepsilon_r^{\bm q} ,\nabla w)_{\mathcal{T}_h}+\left\langle\varepsilon_r^{\bm q}\cdot \bm{n} + \tau (\varepsilon_r^u-\varepsilon_r^{\widehat u}),w\right\rangle_{\partial {\mathcal{T}_h}}\nonumber\\
			&= (\bm q_h - \bm\Pi\bm{q}_h,\nabla w)_{\mathcal{T}_h} + \left\langle (\bm\Pi {\bm{q}}_h - \bm q_h)\cdot \bm{n},w\right\rangle_{\partial {\mathcal{T}_h}}+\left\langle\tau (\Pi u_h-u_h),w\right\rangle_{\partial \mathcal T_h}\nonumber\\
			& \quad - \left\langle\tau (P_M \widehat u_h - \widehat u_h),w\right\rangle_{\partial \mathcal T_h},\label{POD_error_b}\\
			\hspace{2em}&\hspace{-2em} \left\langle\varepsilon_r^{\bm q}\cdot \bm{n} + \tau (\varepsilon_r^u-\varepsilon_r^{\widehat u}),\mu\right\rangle_{\partial {\mathcal{T}_h}}\nonumber\\
			&= \left\langle (\bm \Pi\bm q_h -\bm q_h)\cdot \bm n,\mu\right\rangle_{\partial{\mathcal{T}_h}} +\left\langle \tau(\Pi u_h -u_h), \mu\right\rangle_{\partial{\mathcal{T}_h}},\label{POD_error_c}\\
			\varepsilon_r^u(0) &= \Pi u_h(\cdot,0) - u_{r,0},
		\end{align}
	\end{subequations}
	for all $\bm{v} \in \bm{V}_h^{r_1}$, $w\in W_h^{r_2}$, and $\mu\in M_h^{r_3}$.
\end{lemma}

\begin{proof}
	Noting that $\mu= 0$ on $\varepsilon_h^{\partial}$, the HDG solution $\bm q_h $, $u_h$, and $\widehat u_h$ satisfies
	\begin{equation*}
		\begin{split}
			(c\bm{q}_h,\bm{v})_{\mathcal{T}_h}-(u_h,\nabla\cdot \bm{v})_{\mathcal{T}_h}+\left\langle\widehat{u}_h,\bm{v\cdot n} \right\rangle_{\partial {\mathcal{T}_h}}&= 0, \\
			(\partial_tu_h,w)_{\mathcal T_h}-(\bm{q}_h, \nabla w)_{\mathcal{T}_h}+\left\langle \bm q_h\cdot \bm n,w\right\rangle_{\partial {\mathcal{T}_h}}+\left\langle\tau(u_h-\widehat u_h),w\right\rangle_{\partial {\mathcal{T}_h}} &= (f,w)_{\mathcal{T}_h},\\
			\left\langle \bm q_h\cdot \bm n +\tau(u_h-\widehat u_h), \mu\right\rangle_{\partial{\mathcal{T}_h}} &=0,
		\end{split}
	\end{equation*}
	for all $\bm{v} \in \bm{V}_h^{r_1}$, $w\in W_h^{r_2}$, $\mu\in M_h^{r_3}$. By the definition of the POD projections, we obtain 
	\begin{align*}
		\hspace{2em}&\hspace{-2em} (c\bm\Pi \bm{q}_h,\bm{v})_{\mathcal{T}_h}-(\Pi u_h,\nabla\cdot \bm{v})_{\mathcal{T}_h}+\left\langle P_M \widehat{u}_h,\bm{v\cdot n} \right\rangle_{\partial {\mathcal{T}_h}}\\
		&= 
		(c\bm\Pi \bm{q}_h -c \bm{q}_h ,\bm{v})_{\mathcal{T}_h} + (u_h - \Pi u_h,\nabla\cdot \bm{v})_{\mathcal{T}_h} + \left\langle P_M \widehat{u}_h -\widehat u_h,\bm{v\cdot n} \right\rangle_{\partial {\mathcal{T}_h}},\\
		\hspace{2em}&\hspace{-2em} (\partial_t\Pi u_h,w)_{\mathcal T_h}-(\bm\Pi\bm{q}_h,\nabla w)_{\mathcal{T}_h}+\left\langle \bm\Pi {\bm{q}}_h\cdot \bm{n},w\right\rangle_{\partial {\mathcal{T}_h}}+\left\langle\tau \Pi u_h,w\right\rangle_{\partial \mathcal T_h} \\
		&- \left\langle\tau P_M \widehat u_h,w\right\rangle_{\partial \mathcal T_h}= (f,w)_{\mathcal{T}_h}+(\bm q_h - \bm\Pi\bm{q}_h,\nabla w)_{\mathcal{T}_h} + \left\langle (\bm\Pi {\bm{q}}_h - \bm q_h)\cdot \bm{n},w\right\rangle_{\partial {\mathcal{T}_h}}\\
		&+\left\langle\tau (\Pi u_h-u_h),w\right\rangle_{\partial \mathcal T_h}- \left\langle\tau (P_M \widehat u_h - \widehat u_h),w\right\rangle_{\partial \mathcal T_h},\\
		\hspace{2em}&\hspace{-2em} \left\langle \bm \Pi\bm q_h\cdot \bm n +\tau(\Pi u_h-P_M\widehat u_h), \mu\right\rangle_{\partial{\mathcal{T}_h}}\\
		&= \left\langle (\bm \Pi\bm q_h -\bm q_h)\cdot \bm n +\tau(\Pi u_h -u_h), \mu\right\rangle_{\partial{\mathcal{T}_h}}.
	\end{align*}
	Subtracting \eqref{Heat_HDG_a}-\eqref{Heat_HDG_c} from the above three equations gives the error equations.
\end{proof}

Next, we obtain an expression for the energy norm of the errors.
\begin{lemma}
	We have 
	\begin{align*}
		\hspace{2em}&\hspace{-2em} (\partial_t\varepsilon^u_r,\varepsilon^u_r)_{\mathcal T_h}+( c\varepsilon_r^{\bm q},\varepsilon_r^{\bm q})_{\mathcal T_h} +  \left\langle\tau( \varepsilon_r^{u} - \varepsilon_r^{\widehat u}),  \varepsilon_r^{u} - \varepsilon_r^{\widehat u} \right\rangle_{\partial {\mathcal{T}_h}}\\
		&= (c\bm\Pi \bm{q}_h -c \bm{q}_h , \varepsilon_r^{\bm q})_{\mathcal{T}_h} +\left\langle (\bm\Pi {\bm{q}_h} - \bm q_h)\cdot \bm{n}, \varepsilon_r^u - \varepsilon_r^{\widehat u} \right\rangle_{\partial {\mathcal{T}_h}} 
		\\
		&\quad + \left\langle \tau(\Pi {{u}_h} - u_h), \varepsilon_r^u - \varepsilon_r^{\widehat u}  \right\rangle_{\partial {\mathcal{T}_h}} - \left\langle \tau(P_M {\widehat{u}_h} - \widehat u_h), \varepsilon_r^u \right\rangle_{\partial {\mathcal{T}_h}}\\
		&\quad+\left\langle  u_h - \Pi u_h,  \varepsilon_r^{\bm q} \cdot \bm n\right\rangle_{\partial \mathcal T_h}
		- 	(\nabla (u_h - \Pi u_h), \varepsilon_r^{\bm q})_{\mathcal{T}_h}+ \left\langle P_M \widehat{u}_h -\widehat u_h,\varepsilon_r^{\bm q} \cdot\bm n \right\rangle_{\partial {\mathcal{T}_h}} \\
		& \quad +\left\langle (\bm q_h - \bm\Pi\bm{q}_h) \cdot \bm n, \varepsilon_r^u\right\rangle_{\partial\mathcal{T}_h} 
		-  (\nabla \cdot(\bm q_h - \bm\Pi\bm{q}_h), \varepsilon_r^u)_{\mathcal{T}_h}.
	\end{align*}
\end{lemma}

\begin{proof}
	Taking $\bm v = \varepsilon_r^{\bm q}$ in \eqref{POD_error_a}, $w = \varepsilon_r^u$ in \eqref{POD_error_b},  $\mu = -\varepsilon_r^{\widehat u}$ in \eqref{POD_error_c}, and adding the resulting four equations gives
	\begin{align*}
		\hspace{2em}&\hspace{-2em} (\partial_t\varepsilon^u_r,\varepsilon^u_r)_{\mathcal T_h}+( c\varepsilon_r^{\bm q},\varepsilon_r^{\bm q})_{\mathcal T_h} +\Theta_r\\
		&= (c\bm\Pi \bm{q}_h -c \bm{q}_h , \varepsilon_r^{\bm q})_{\mathcal{T}_h}+ \left\langle (\bm\Pi {\bm{q}_h} - \bm q_h)\cdot \bm{n}, \varepsilon_r^u - \varepsilon_r^{\widehat u} \right\rangle_{\partial {\mathcal{T}_h}} \\
		& \quad + \left\langle \tau(\Pi {{u}_h} - u_h), \varepsilon_r^u - \varepsilon_r^{\widehat u}  \right\rangle_{\partial {\mathcal{T}_h}}- \left\langle \tau(P_M {\widehat{u}_h} - \widehat u_h), \varepsilon_r^u \right\rangle_{\partial {\mathcal{T}_h}} \\
		&\quad + (u_h - \Pi u_h,\nabla\cdot \varepsilon_r^{\bm q})_{\mathcal{T}_h} + \left\langle P_M \widehat{u}_h -\widehat u_h,\varepsilon_r^{\bm q} \cdot\bm n \right\rangle_{\partial {\mathcal{T}_h}}+ (\bm q_h - \bm\Pi\bm{q}_h,\nabla \varepsilon_r^u)_{\mathcal{T}_h},
	\end{align*}
	where 
	\begin{align*}
		\Theta_r &=- (\varepsilon_r^u, \nabla\cdot \varepsilon_r^{\bm q})_{\mathcal T_h} - (\varepsilon_r^{\bm q} ,\nabla \varepsilon_r^u)_{\mathcal{T}_h} + \left\langle\varepsilon_r^{\bm q}\cdot \bm{n} + \tau (\varepsilon_r^u-\varepsilon_r^{\widehat u}),\varepsilon_r^u \right\rangle_{\partial {\mathcal{T}_h}}\\
		&\quad -
		\left\langle  \tau (\varepsilon_r^u-\varepsilon_r^{\widehat u}),\varepsilon_r^{\widehat u} \right\rangle_{\partial {\mathcal{T}_h}}\\
		& = \left\langle  \tau (\varepsilon_r^u-\varepsilon_r^{\widehat u}),\varepsilon_r^u-\varepsilon_r^{\widehat u} \right\rangle_{\partial {\mathcal{T}_h}}.
	\end{align*}
	Moreover,
	\begin{align*}
		(u_h - \Pi u_h,\nabla\cdot \varepsilon_r^{\bm q})_{\mathcal{T}_h} &= \left\langle  u_h - \Pi u_h,  \varepsilon_r^{\bm q} \cdot \bm n\right\rangle_{\partial \mathcal T_h} - 	(\nabla (u_h - \Pi u_h), \varepsilon_r^{\bm q})_{\mathcal{T}_h},\\
		(\bm q_h - \bm\Pi\bm{q}_h,\nabla \varepsilon_r^u)_{\mathcal{T}_h}&= \left\langle (\bm q_h - \bm\Pi\bm{q}_h) \cdot \bm n, \varepsilon_r^u\right\rangle_{\partial\mathcal{T}_h} -  (\nabla \cdot(\bm q_h - \bm\Pi\bm{q}_h), \varepsilon_r^u)_{\mathcal{T}_h}.
	\end{align*}
\end{proof}

Next, we obtain the main energy estimate.
\begin{theorem}\label{immediate}
	For any $t>0$, we have 
	\begin{align*}
		\norm{\varepsilon^u_r(t)}^2_{\mathcal T_h} +\int_0^t Z(s)ds
		\le A(t) + \int_0^t B(s) \norm{\varepsilon^u_r}_{\mathcal T_h} ds,
	\end{align*}
	where 
	\begin{align*}
		Z &= c_0\norm{ \varepsilon_h^{\bm q}}_{\mathcal T_h}^2+\tau_* \norm{\varepsilon_h^{\widehat u} - \varepsilon_h^{u}}_{\partial {\mathcal{T}_h}}^2,\\
		A &= \norm{\Pi u_h(\cdot,0) - u_{r,0}}_{\mathcal T_h}^2 + \int_0^t   \frac{16c_1^2}{c_0} \norm {\bm\Pi \bm{q}_h - \bm{q}_h}_{\mathcal T_h}^2 + \frac{8}{\tau_*}\norm{(\bm\Pi {\bm{q}_h} - \bm q_h) \cdot \bm n}_{\partial {\mathcal{T}_h}}^2 ds\\
		& \quad + \int_0^t \left(\frac{8 (\tau^*)^2}{\tau_*} +\frac{16C}{c_0h} \right)\norm{\Pi {u}_h - u_h}_{\partial{\mathcal{T}_h}}^2 + \frac{16}{c_0}\norm{\nabla (u_h - \Pi u_h)}_{\mathcal T_h}^2 \\
		&\quad+\frac{16C}{c_0h} \norm{P_M \widehat u_h-  \widehat u_h}_{{\partial\mathcal{T}_h}}^2 ds,\\
		B &= 2\left(C^{1/2}\norm {(\bm q_h - \bm\Pi\bm{q}_h) \cdot \bm n}_{\partial \mathcal T_h} h^{-1/2} +  \norm {\nabla \cdot(\bm q_h - \bm\Pi\bm{q}_h)}_{\mathcal{T}_h}\right.\\
		&\qquad \quad\left.+C^{1/2} \tau^* \norm{P_M\widehat u_h- \widehat u_h}_{\partial \mathcal T_h} h^{-1/2}\right).
	\end{align*}
	Here, $0<\tau_* = \min_{K\in\mathcal T_h} \tau\le \tau^* = \max_{K\in\mathcal T_h} \tau$, and $C$ is the constant defined in \eqref{tra}.
\end{theorem}

\begin{proof}
	We have
	\begin{align*}
		\hspace{2em}&\hspace{-2em} \frac 1 2 \frac{d}{dt}\norm{\varepsilon^u_r}^2_{\mathcal T_h} +c_0\norm{ \varepsilon_r^{\bm q}}_{\mathcal T_h}^2+\tau_* \norm{\varepsilon_r^{u} - \varepsilon_r^{\widehat u} }_{\partial {\mathcal{T}_h}}^2\\
		&\le 
		c_1 \norm {\bm\Pi \bm{q}_h - \bm{q}_h}_{\mathcal T_h} \norm{\varepsilon_r^{\bm q}}_{\mathcal{T}_h} +\norm{(\bm q_h - \bm\Pi\bm{q}_h) \cdot \bm n}_{\partial {\mathcal{T}_h}} \norm{\varepsilon_r^{u} - \varepsilon_r^{\widehat u}}_{\partial {\mathcal{T}_h}}\\
		& \quad +
		\tau^* \norm{\Pi {u}_h - u_h}_{\partial {\mathcal{T}_h}}
		\norm{\varepsilon_r^{u} - \varepsilon_r^{\widehat u}}_{\partial {\mathcal{T}_h}} +\tau^* \norm{P_M\widehat u_h- \widehat u_h}_{\partial \mathcal T_h} \norm{\varepsilon_r^{u}}_{\partial\mathcal T_h}\\
		& \quad + \norm{\Pi {{u}_h} -  u_h}_{\partial{\mathcal{T}_h}} \norm{\varepsilon_r^{\bm q}}_{\partial\mathcal T_h} + \norm{\nabla (u_h - \Pi u_h)}_{\mathcal{T}_h} \norm {\varepsilon_r^{\bm q}}_{\mathcal{T}_h} \\
		&\quad +\norm{P_M \widehat u_h-  \widehat u_h}_{{\partial\mathcal{T}_h}}  \norm{\varepsilon_r^{\bm q}}_{\partial \mathcal T_h} +\norm {(\bm q_h - \bm\Pi\bm{q}_h) \cdot \bm n}_{\partial \mathcal T_h}  \norm {\varepsilon_r^u}_{\partial\mathcal{T}_h} \\
		&\quad+  \norm {\nabla \cdot(\bm q_h - \bm\Pi\bm{q}_h)}_{\mathcal{T}_h} \norm {\varepsilon_r^u}_{\mathcal{T}_h}\\
		&\le \frac{8c_1^2}{c_0} \norm {\bm\Pi \bm{q}_h - \bm{q}_h}_{\mathcal T_h}^2 +\frac {c_0}{8}\norm{\varepsilon_r^{\bm q}}_{\mathcal{T}_h}^2 +\frac{4}{\tau_*}\norm{(\bm q_h - \bm\Pi\bm{q}_h) \cdot \bm n}_{\partial {\mathcal{T}_h}}^2 \\
		&\quad +  \frac {\tau_*} {4} \norm{\varepsilon_r^{u} - \varepsilon_r^{\widehat u}}_{\partial {\mathcal{T}_h}}^2 +\frac{4 (\tau^*)^2}{\tau_*}\norm{\Pi {u}_h - u_h}_{\partial{\mathcal{T}_h}}^2 + \frac {\tau_*} {4} \norm{\varepsilon_r^{u} - \varepsilon_r^{\widehat u}}_{\partial {\mathcal{T}_h}}^2 \\
		&\quad + C^{1/2}\tau^* \norm{P_M\widehat u_h- \widehat u_h}_{\partial \mathcal T_h} h^{-1/2} \norm{\varepsilon_r^{u}}_{\mathcal{T}_h}+ \frac{8C}{c_0h}\norm{\Pi {{u}_h} -  u_h}_{\partial{\mathcal{T}_h}}^2 \\
		&\quad+ \frac{c_0}{8} \norm{\varepsilon_r^{\bm q}}_{\mathcal T_h}^2 + \frac{8}{c_0}\norm{\nabla (u_h - \Pi u_h)}_{\mathcal T_h}^2 + \frac{c_0}{8} \norm {\varepsilon_r^{\bm q}}_{\mathcal{T}_h}^2\\
		& \quad +\frac{8C}{c_0h} \norm{P_M \widehat u_h-  \widehat u_h}_{{\partial\mathcal{T}_h}}^2 +  \frac{c_0}{8} \norm{\varepsilon_r^{\bm q}}_{\mathcal T_h}^2 \\
		&\quad+C^{1/2}\norm {(\bm q_h - \bm\Pi\bm{q}_h) \cdot \bm n}_{\partial \mathcal T_h} h^{-1/2} \norm {\varepsilon_r^u}_{\mathcal{T}_h}+\norm {\nabla \cdot(\bm q_h - \bm\Pi\bm{q}_h)}_{\mathcal{T}_h} \norm {\varepsilon_r^u}_{\mathcal{T}_h},
	\end{align*}
	which means
	\begin{align*}
		\hspace{1em}&\hspace{-1em} \frac{d}{dt}\norm{\varepsilon^u_r}^2_{\mathcal T_h} +{c_0}\norm{ \varepsilon_r^{\bm q}}_{\mathcal T_h}^2+\tau_* \norm{\varepsilon_r^{u} - \varepsilon_r^{\widehat u} }_{\partial {\mathcal{T}_h}}^2\\
		& \le \frac{16c_1^2}{c_0} \norm {\bm\Pi \bm{q}_h - \bm{q}_h}_{\mathcal T_h}^2 + \frac{8}{\tau_*}\norm{(\bm q_h - \bm\Pi\bm{q}_h) \cdot \bm n}_{\partial {\mathcal{T}_h}}^2 \\
		&\quad+\left(\frac{8 (\tau^*)^2}{\tau_*} +\frac{16C}{c_0h}\right)\norm{\Pi {u}_h - u_h}_{\partial{\mathcal{T}_h}}^2\\
		& \quad + \frac{16}{c_0}\norm{\nabla (u_h - \Pi u_h)}_{\mathcal T_h}^2 
		+\frac{16C}{c_0h} \norm{P_M \widehat u_h-  \widehat u_h}_{{\partial\mathcal{T}_h}}^2\\
		& \quad +2 C^{1/2}\norm {(\bm q_h - \bm\Pi\bm{q}_h) \cdot \bm n}_{\partial \mathcal T_h} h^{-1/2} \norm{\varepsilon_r^{u}}_{\mathcal{T}_h}   + 2 \norm {\nabla \cdot(\bm q_h - \bm\Pi\bm{q}_h)}_{\mathcal{T}_h} \norm{\varepsilon_r^{u}}_{\mathcal{T}_h}\\
		& \quad + 2 C^{1/2}\tau^* \norm{P_M\widehat u_h- \widehat u_h}_{\partial \mathcal T_h} h^{-1/2} \norm{\varepsilon_r^{u}}_{\mathcal{T}_h}.
	\end{align*}
	Integrating in time over the interval $(0,t)$ and using $\varepsilon_r^u(0) = \Pi u_h(\cdot,0) - u_{r,0}$ gives the result.
\end{proof}

To estimate the errors, we use the following result whose proof can be found in \cite{chabaud2012uniform}.
\begin{lemma}\label{cock_in}
	Assume for all $t>0$ we have 
	\begin{align*}
		\eta^2(t) +\int_0^t Z(s)ds \le A(t)+ \int_0^t B(s)\eta(s) ds,
	\end{align*}
	for some nonnegative functions $A,B, Z$ in $L^{\infty}(\mathbb R^+).$ Then, for any $t>0$,
	\begin{align*}
		\eta^2(t) +\int_0^t Z(s)ds \le \left( \left[ \max_{0\le s\le t}A(s) \right]^{1/2} +\frac 1 2 \int_0^t B(s)ds\right)^2.
	\end{align*}
\end{lemma}

We are now ready to prove the main error analysis result.  We note that we do not attempt to carefully track the constants appearing in the error bounds.  After the proof, we discuss some features of the error bounds.
\begin{theorem}
	Let the main assumption hold.  If $(\bm q_h, u_h)$ and $(\bm q_r, u_r)$ are the HDG solution of the heat equation \eqref{HDG_discrete2} and the HDG-POD reduced order model \eqref{Heat}, respectively, then
	\begin{align*}
		\int_0^T \norm{u_h- u_r}_{\mathcal T_h}^2dt  &\le  C_1(T) \norm{\Pi u_h(\cdot,0) - u_{r,0}}_{\mathcal T_h}^2 + \frac{C_2(T)}{h} \Lambda_r^u,\\
		\int_0^T \norm{ \bm q_h- \bm q_r}_{\mathcal T_h}^2dt &  \le  C_3(T) \norm{\Pi u_h(\cdot,0) - u_{r,0}}_{\mathcal T_h}^2 + \frac{C_4(T)}{h} \Lambda_r^{\bm q},
	\end{align*}
	where
	\begin{align*}
		\Lambda_r^u &= \sum_{i>r_1}  \lambda_i^{\bm q}(1+\|\varphi_i^{\bm q} \cdot \bm n\|_{\partial \mathcal T_h}^2+ \| \nabla \cdot \varphi_i^{\bm q} \|_{\mathcal T_h}^2)  \\
		&\quad+\sum_{i>r_2}  \lambda_i^{u}(1+\|\varphi_i^{u}\|_{\partial \mathcal T_h}^2+\| \nabla \varphi_i^{u} \|_{ \mathcal T_h}^2) + \sum_{i>r_3} \lambda_i^{\widehat u},\\
		\Lambda_r^{\bm q} &= \sum_{i>r_1}  \lambda_i^{\bm q}(1+\|\varphi_i^{\bm q} \cdot \bm n\|_{\partial \mathcal T_h}^2+\| \nabla \cdot \varphi_i^{\bm q} \|_{\mathcal T_h}^2)  \\
		&\quad+\sum_{i>r_2}  \lambda_i^{u}(\|\varphi_i^{u}\|_{\partial \mathcal T_h}^2+\| \nabla \varphi_i^{u} \|_{ \mathcal T_h}^2) + \sum_{i>r_3} \lambda_i^{\widehat u},
	\end{align*}
	and $C_1,C_2,C_3,C_4$ are constants independent of the reduced order $r$ and the mesh spacing $ h $. Furthermore, for $ h $ fixed, if $\norm{\Pi u_h(\cdot,0) - u_{r,0}}_{\mathcal T_h}\to 0$ as $r$ increases, then $u_r\to u_h$ in $L^2(0,T;L^2(\Omega))$ and  $\bm q_r\to \bm q_h$ in $ L^2(0,T;[L^2(\Omega)]^d)$.
\end{theorem}

\begin{proof}
	For $t \in(0,T]$, we have 
	\begin{align*}
		\norm{\varepsilon^u_r(t)}^2_{\mathcal T_h} +\int_0^t Z(s)ds
		\le A(t) + \int_0^t B(s) \norm{\varepsilon^u_r}_{\mathcal T_h} ds,
	\end{align*}
	where $Z,A,B$ are defined in Theorem \ref{immediate}.  Apply Lemma \ref{cock_in} and Jensen's inequality to get 
	\begin{align*}
		\norm{\varepsilon_r^u(t)}_{\mathcal T_h}^2+\int_0^t Z(s)ds &\le \left( A^{1/2}(t) +\frac 1 2 \int_0^t B(s)ds\right)^2\\
		&\le  2 A(t) + t \int_0^t B^2(s)ds\\
		& \le 2 \norm{\Pi u_h(\cdot,0) - u_{r,0}}_{\mathcal T_h}^2 + \frac{C(t)}{h} \int_0^t \theta^2(s) ds,
	\end{align*}
	where 
	\begin{align*}
		\theta^2(s) &= \norm{\bm\Pi {\bm{q}_h} - \bm q_h}_{ {\mathcal{T}_h}}^2  + \norm{(\bm\Pi {\bm{q}_h} - \bm q_h) \cdot \bm n}_{ \partial{\mathcal{T}_h}}^2  +  \norm{\nabla \cdot (\bm\Pi {\bm{q}_h} - \bm q_h)}_{{\mathcal{T}_h}}^2\\
		&\quad + \norm{\Pi {{u}_h} - u_h}_{{\partial \mathcal{T}_h}}^2
		+\norm{\nabla (\Pi {{u}_h} - u_h)}_{{\mathcal{T}_h}}^2+ \norm{P_M {\widehat{u}_h} - \widehat u_h}_{{\partial\mathcal{T}_h}}^2.
	\end{align*}
	By the POD projection error results, we get
	\begin{align*}
		\int_0^t \theta^2(s) ds \le  \int_0^T \theta^2(s) ds = \Lambda_r^{\bm q}.
	\end{align*}
	Therefore, we have 
	\begin{align*}
		\norm{\varepsilon_r^u(t)}_{\mathcal T_h}^2  \le 2 \norm{\Pi u_h(\cdot,0) - u_{r,0}}_{\mathcal T_h}^2 + \frac{C(t)}{h}  \Lambda_r^{\bm q}.
	\end{align*}
	Since $u_h - u_r = u_h-\Pi u_h +\Pi u_h - u_r = u_h-\Pi u_h +\varepsilon_h^u$, we have
	\begin{align*}
		\int_0^T \norm{u_h- u_r}_{\mathcal T_h}^2dt  &\le \int_0^T \norm{u_h-\Pi u_h}_{\mathcal T_h}^2dt +\int_0^T \norm{\varepsilon_r^u}_{\mathcal T_h}^2dt\\
		& \le  \sum_{i > r_2} \lambda_i^u + 2 T \norm{\Pi u_h(\cdot,0) - u_{r,0}}_{\mathcal T_h}^2  + \frac{C(T)}{h}  \Lambda_r^{\bm q}.
	\end{align*}
	This proves the first error bound; a similar argument gives the estimate for $\bm q_h - \bm q_r$.
	%
\end{proof}

\begin{remark}
	For the initial condition $u_{r,0} = \Pi u_h(\cdot, 0)$, the first term in each error bound equals zero.
\end{remark}

In the error bounds for both $ u_h $ and $ \bm q_h $, the bounds contain a factor of $ h^{-1/2} $ multiplying the POD projection error terms $ (\Lambda_r^u)^{1/2} $ and $ (\Lambda_r^{\bm q})^{1/2} $.  This $ h^{-1/2} $ factor increases slowly as the mesh is refined.  Since the POD eigenvalues often decay very quickly as $ r $ increases, the $ h^{-1/2} $ factor will typically be dominated by the rapid decay of the POD projection error terms $ (\Lambda_r^u)^{1/2} $ and $ (\Lambda_r^{\bm q})^{1/2} $ in the error bound.

In the HDG error analysis for the heat equation in \cite{chabaud2012uniform}, a special HDG projection is used to optimize the errors bounds and obtain optimal convergence orders for the discretization in space.  It is unclear if such a special projection can be utilized within a POD framework to optimize the resulting error bounds.  We leave this to be considered elsewhere.

\section{Numerical Results}
\label{sec:numerical_results}

We report numerical results for one 2D problem and one 3D problem.  For both examples, we use backward Euler for the time discretization and use time step size $\Delta t = 0.001$ for $ 0 \leq t \leq 1 $.  We also chose $ \tau = 1 $ for the HDG stabilization function.  We report computed errors for approximations to the time integrals from our continuous time error estimates.  Specifically, we compute the errors
\begin{align*}
	\mbox{$ \bm q $ error} &= \left( \frac{1}{N}\sum_{i=1}^N \norm{\bm q_h(\cdot,t_i)-\bm{q}_r(\cdot,t_i)}_{\mathcal T_h}^2 \right)^{1/2},\\
	\mbox{$ u $ error} &= \left( \frac{1}{N}\sum_{i=1}^N \norm{ u_h(\cdot,t_i)-u_r(\cdot,t_i)}_{\mathcal T_h}^2 \right)^{1/2}.
\end{align*}

For the 2D example, we take problem data $ u_0 =\sin(\pi x)\sin(\pi y)e^{x}\cos(y)$, $c=0.01 $, and $ f = 0 $ with domain $\Omega = [0,1]\times [0,1]$.  For the 3D example, we take problem data $c=0.01 $, $ f = 0 $, and $ u_0 = \sin(\pi x)\sin(\pi y)\sin(\pi z)e^{x}\cos(y)z $ with domain $\Omega = [0,1]\times [0,1] \times [0,1]$.  The 2D domain is partitioned into 4096 triangles, and the 3D domain is partitioned into 24576 tetrahedra.

\Cref{fig:POD_svalues} shows the POD singular values for both examples.  The POD singular values decay very rapidly for all variables, and they decay at similar rates.  It is interesting to note that the POD singular values for the scalar variable $ u_h $ and the flux $ \bm q_h $ are numerically very similar, while the POD singular values for the numerical trace $ \widehat{u}_h $ are at least an order of magnitude larger for both examples.
\begin{figure}[htb]
	\includegraphics[width=.45\linewidth]{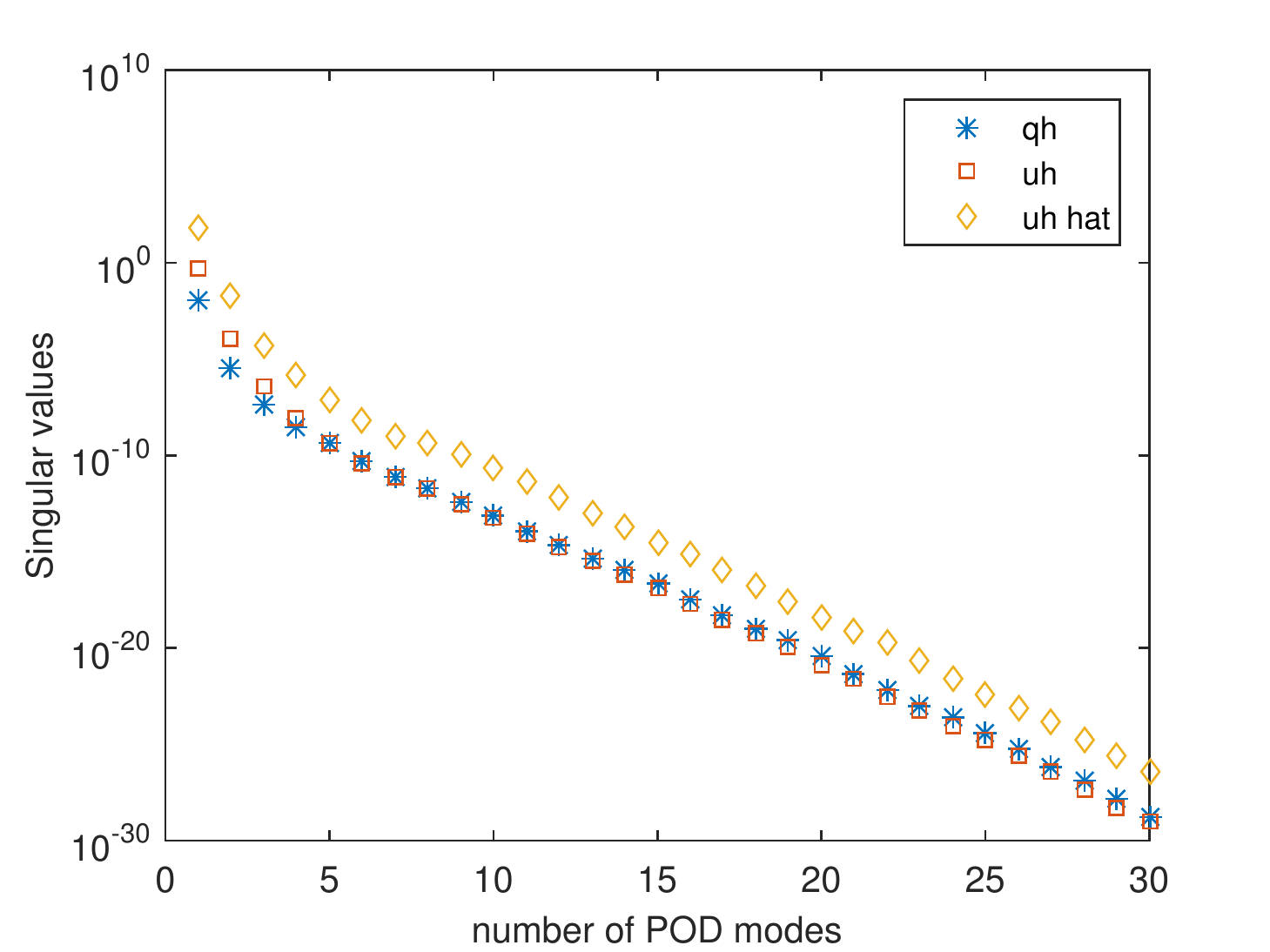}
	\includegraphics[width=.45\linewidth]{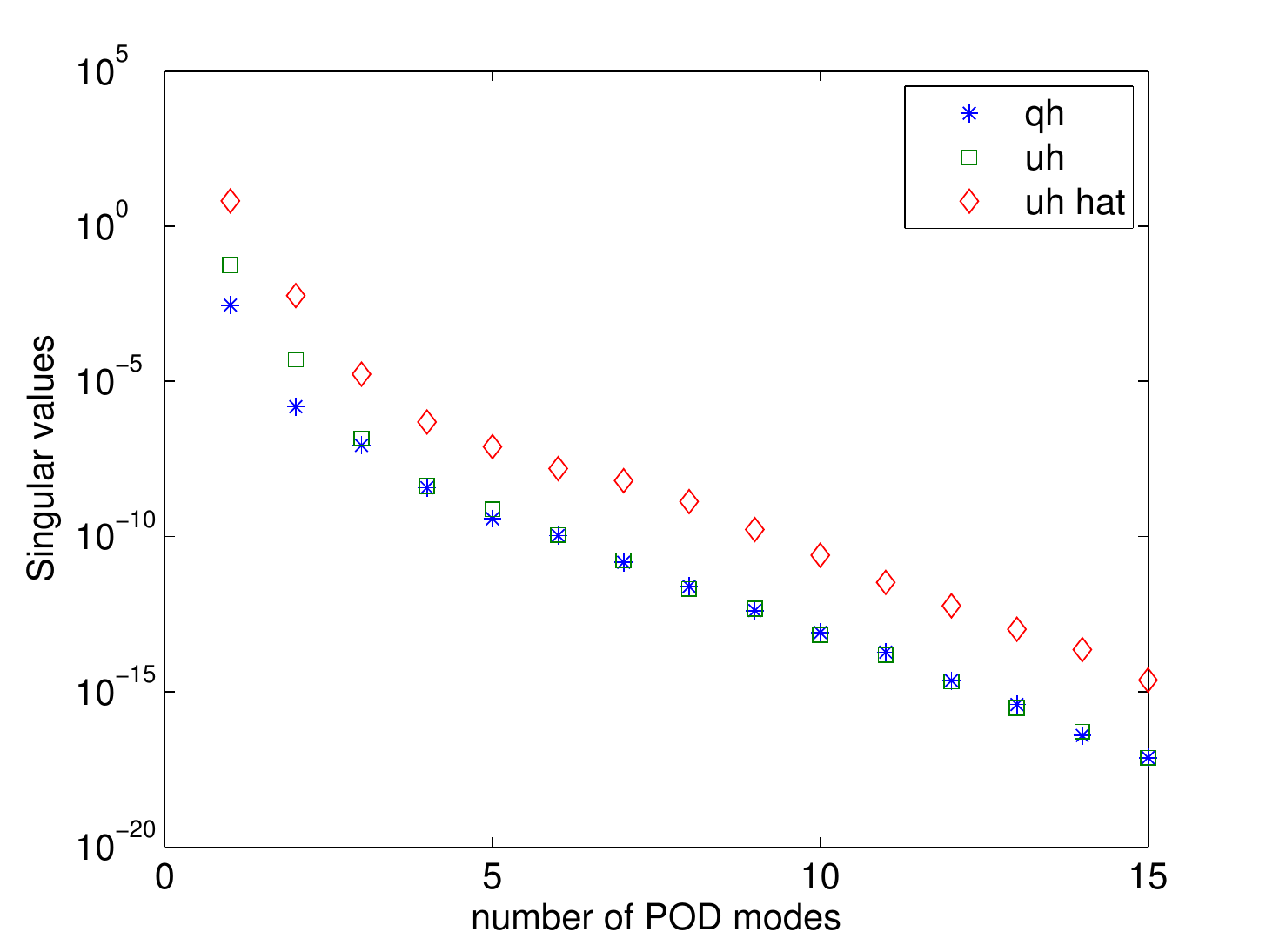}
	\caption{\label{fig:POD_svalues}POD singular values: 2D example (left) and 3D example (right)}
\end{figure}

\Cref{table:2D_example} and \Cref{table:3D_example} show the computed errors for the 2D and 3D examples, respectively.  As predicted by the error bounds, the errors decay rapidly to zero as $ r $ increases.
\begin{table}[!hbp]
	\begin{center}
		\begin{tabular}{|c|c|c|c|c|c|}
			\hline
			r &7& 10&13 &16 & 20 \\
			\hline
			$ \bm q $ error &1.782E-06& 1.670E-07&1.271E-08& 9.979E-10 & 2.940E-11 \\
			\hline
			$ u $ error &1.914E-06& 1.767E-07&1.290E-08& 8.569E-10 & 2.319E-11 \\
			\hline
		\end{tabular}
	\end{center}
	\caption{\label{table:2D_example}2D example: Computed errors}
\end{table}
\begin{table}[!hbp]
	\begin{center}
		\begin{tabular}{|c|c|c|c|c|c|}
			\hline
			r &3& 6&9 &12 & 15 \\
			\hline
			$ \bm q $ error &6.801E-05& 4.933E-06&3.941E-07& 2.363E-08 & 1.323E-09 \\
			\hline
			$ u $ error &1.434E-04& 7.048E-06&4.547E-07& 2.711E-08 & 2.090E-09 \\
			\hline
		\end{tabular}
	\end{center}
	\caption{\label{table:3D_example}3D example: Computed errors}
\end{table}

\section{Conclusion}  We proposed a new model reduction procedure combining HDG and POD for the heat equation.  The HDG-POD procedure is unlike other typical POD-based reduced order models since it does not rely on a standard Galerkin projection framework.  We provided an error analysis for the method, and showed the error bounds converge to zero as the order $ r $ of the reduced model increases.  We presented numerical results for 2D and 3D example problems that illustrated the theory.

This work contains the initial study of the HDG-POD model reduction framework.  There are many avenues that remain to be investigated.  It would be interesting to see if the error analysis can be improved.  Also, as mentioned earlier, we intend to investigate HDG-POD for more complex nonlinear PDEs.  Furthermore, detailed numerical experiments would be of interest to see how HDG-POD compares to other model reduction schemes.

\section*{Acknowledgements}  The authors thank Bernardo Cockburn for many helpful conversations.

\bibliographystyle{plain}
\bibliography{yangwen_ref_papers,yangwen_ref_books}

\begin{thebibliography}{10}

\bibitem{akman2016error}
Tu\u{g}ba Akman.
\newblock Error estimates for space-time discontinuous {G}alerkin formulation
  based on proper orthogonal decomposition.
\newblock {\em Appl. Anal.}, 96(3):461--482, 2017.

\bibitem{amsallem2014error}
D.~Amsallem and U.~Hetmaniuk.
\newblock Error estimates for {G}alerkin reduced-order models of the
  semi-discrete wave equation.
\newblock {\em ESAIM Math. Model. Numer. Anal.}, 48(1):135--163, 2014.

\bibitem{arnold2000discontinuous}
Douglas~N. Arnold, Franco Brezzi, Bernardo Cockburn, and Donatella Marini.
\newblock Discontinuous {G}alerkin methods for elliptic problems.
\newblock In {\em Discontinuous {G}alerkin methods ({N}ewport, {RI}, 1999)},
  volume~11 of {\em Lect. Notes Comput. Sci. Eng.}, pages 89--101. Springer,
  Berlin, 2000.

\bibitem{arnold2002unified}
Douglas~N. Arnold, Franco Brezzi, Bernardo Cockburn, and L.~Donatella Marini.
\newblock Unified analysis of discontinuous {G}alerkin methods for elliptic
  problems.
\newblock {\em SIAM J. Numer. Anal.}, 39(5):1749--1779, 2001/02.

\bibitem{atwell2001reduced}
Jeanne~A. Atwell, Jeffrey~T. Borggaard, and Belinda~B. King.
\newblock Reduced order controllers for {B}urgers' equation with a nonlinear
  observer.
\newblock {\em Int. J. Appl. Math. Comput. Sci.}, 11(6):1311--1330, 2001.

\bibitem{BakerGallivan12}
C.~G. Baker, K.~A. Gallivan, and P.~Van~Dooren.
\newblock {L}ow-rank incremental methods for computing dominant singular
  subspaces.
\newblock {\em Linear Algebra Appl.}, 436(8):2866--2888, 2012.

\bibitem{banks2000nondestructive}
H.~T. Banks, Michele~L. Joyner, Buzz Wincheski, and William~P. Winfree.
\newblock Nondestructive evaluation using a reduced-order computational
  methodology.
\newblock {\em Inverse Problems}, 16(4):929--945, 2000.

\bibitem{bassi1997high}
F.~Bassi and S.~Rebay.
\newblock A high-order accurate discontinuous finite element method for the
  numerical solution of the compressible {N}avier-{S}tokes equations.
\newblock {\em J. Comput. Phys.}, 131(2):267--279, 1997.

\bibitem{Brand06}
Matthew Brand.
\newblock {F}ast low-rank modifications of the thin singular value
  decomposition.
\newblock {\em Linear Algebra Appl.}, 415(1):20--30, 2006.

\bibitem{Bui-Thanh16}
Tan Bui-Thanh.
\newblock Construction and analysis of {HDG} methods for linearized shallow
  water equations.
\newblock {\em SIAM J. Sci. Comput.}, 38(6):A3696--A3719, 2016.

\bibitem{castillo2000priori}
Paul Castillo, Bernardo Cockburn, Ilaria Perugia, and Dominik Sch\"otzau.
\newblock An a priori error analysis of the local discontinuous {G}alerkin
  method for elliptic problems.
\newblock {\em SIAM J. Numer. Anal.}, 38(5):1676--1706, 2000.

\bibitem{chabaud2012uniform}
Brandon Chabaud and Bernardo Cockburn.
\newblock Uniform-in-time superconvergence of {HDG} methods for the heat
  equation.
\newblock {\em Math. Comp.}, 81(277):107--129, 2012.

\bibitem{chapelle2013galerkin}
D.~Chapelle, A.~Gariah, P.~Moireau, and J.~Sainte-Marie.
\newblock A {G}alerkin strategy with proper orthogonal decomposition for
  parameter-dependent problems---analysis, assessments and applications to
  parameter estimation.
\newblock {\em ESAIM Math. Model. Numer. Anal.}, 47(6):1821--1843, 2013.

\bibitem{chapelle2012galerkin}
Dominique Chapelle, Asven Gariah, and Jacques Sainte-Marie.
\newblock Galerkin approximation with proper orthogonal decomposition: new
  error estimates and illustrative examples.
\newblock {\em ESAIM Math. Model. Numer. Anal.}, 46(4):731--757, 2012.

\bibitem{cockburn2009hybridizable}
Bernardo Cockburn, Bo~Dong, Johnny Guzm\'an, Marco Restelli, and Riccardo
  Sacco.
\newblock A hybridizable discontinuous {G}alerkin method for steady-state
  convection-diffusion-reaction problems.
\newblock {\em SIAM J. Sci. Comput.}, 31(5):3827--3846, 2009.

\bibitem{cockburn2009unified}
Bernardo Cockburn, Jayadeep Gopalakrishnan, and Raytcho Lazarov.
\newblock Unified hybridization of discontinuous {G}alerkin, mixed, and
  continuous {G}alerkin methods for second order elliptic problems.
\newblock {\em SIAM J. Numer. Anal.}, 47(2):1319--1365, 2009.

\bibitem{cockburn1998local}
Bernardo Cockburn and Chi-Wang Shu.
\newblock The local discontinuous {G}alerkin method for time-dependent
  convection-diffusion systems.
\newblock {\em SIAM J. Numer. Anal.}, 35(6):2440--2463, 1998.

\bibitem{CuiZhang14}
Jintao Cui and Wujun Zhang.
\newblock An analysis of {HDG} methods for the {H}elmholtz equation.
\newblock {\em IMA J. Numer. Anal.}, 34(1):279--295, 2014.

\bibitem{FareedShenSinglerZhang17}
Hiba Fareed, Jiguang Shen, John~R. Singler, and Yangwen Zhang.
\newblock Incremental proper orthogonal decomposition for {PDE} simulation
  data.
\newblock {\em Computers \& Mathematics with Applications}.
\newblock to appear.

\bibitem{del2008error}
Pedro Gal\'an~del Sastre and Rodolfo Bermejo.
\newblock Error estimates of proper orthogonal decomposition eigenvectors and
  {G}alerkin projection for a general dynamical system arising in fluid models.
\newblock {\em Numer. Math.}, 110(1):49--81, 2008.

\bibitem{GaticaSequeira16}
Gabriel~N. Gatica and Fil\'ander~A. Sequeira.
\newblock A priori and a posteriori error analyses of an augmented {HDG} method
  for a class of quasi-{N}ewtonian {S}tokes flows.
\newblock {\em J. Sci. Comput.}, 69(3):1192--1250, 2016.

\bibitem{GaticaSequeira17}
Gabriel~N. Gatica and Fil\'ander~A. Sequeira.
\newblock Analysis of the {HDG} method for the {S}tokes-{D}arcy coupling.
\newblock {\em Numer. Methods Partial Differential Equations}, 33(3):885--917,
  2017.

\bibitem{gubisch2013proper}
Martin Gubisch and Stefan Volkwein.
\newblock Proper orthogonal decomposition for linear-quadratic optimal control.
\newblock In P.~Benner, A.~Cohen, M.~Ohlberger, and K.~Willcox, editors, {\em
  Model Reduction and Approximation: Theory and Algorithms}, pages 5--66. SIAM,
  2017.

\bibitem{herkt2013convergence}
Sabrina Herkt, Michael Hinze, and Rene Pinnau.
\newblock Convergence analysis of {G}alerkin {POD} for linear second order
  evolution equations.
\newblock {\em Electron. Trans. Numer. Anal.}, 40:321--337, 2013.

\bibitem{holmes2012turbulence}
Philip Holmes, John~L. Lumley, Gahl Berkooz, and Clarence~W. Rowley.
\newblock {\em Turbulence, coherent structures, dynamical systems and
  symmetry}.
\newblock Cambridge Monographs on Mechanics. Cambridge University Press,
  Cambridge, second edition, 2012.

\bibitem{homescu2005error}
Chris Homescu, Linda~R. Petzold, and Radu Serban.
\newblock Error estimation for reduced-order models of dynamical systems.
\newblock {\em SIAM J. Numer. Anal.}, 43(4):1693--1714, 2005.

\bibitem{huynh2013high}
L.~N.~T. Huynh, N.~C. Nguyen, J.~Peraire, and B.~C. Khoo.
\newblock A high-order hybridizable discontinuous {G}alerkin method for
  elliptic interface problems.
\newblock {\em Internat. J. Numer. Methods Engrg.}, 93(2):183--200, 2013.

\bibitem{iliescu2014variational}
Traian Iliescu and Zhu Wang.
\newblock Variational multiscale proper orthogonal decomposition:
  {N}avier-{S}tokes equations.
\newblock {\em Numer. Methods Partial Differential Equations}, 30(2):641--663,
  2014.

\bibitem{IwenOng16}
M.~A. Iwen and B.~W. Ong.
\newblock A distributed and incremental {SVD} algorithm for agglomerative data
  analysis on large networks.
\newblock {\em SIAM J. Matrix Anal. Appl.}, 37(4):1699--1718, 2016.

\bibitem{jin2017analysis}
Bangti Jin and Zhi Zhou.
\newblock An analysis of {G}alerkin proper orthogonal decomposition for
  subdiffusion.
\newblock {\em ESAIM Math. Model. Numer. Anal.}, 51(1):89--113, 2017.

\bibitem{Kostova-Vassilevska18}
Tanya Kostova-Vassilevska and Geoffrey~M. Oxberry.
\newblock Model reduction of dynamical systems by proper orthogonal
  decomposition: Error bounds and comparison of methods using snapshots from
  the solution and the time derivatives.
\newblock {\em Journal of Computational and Applied Mathematics},
  330(Supplement C):553 -- 573, 2018.

\bibitem{kunisch2001galerkin}
K.~Kunisch and S.~Volkwein.
\newblock Galerkin proper orthogonal decomposition methods for parabolic
  problems.
\newblock {\em Numer. Math.}, 90(1):117--148, 2001.

\bibitem{kunisch2002galerkin}
K.~Kunisch and S.~Volkwein.
\newblock Galerkin proper orthogonal decomposition methods for a general
  equation in fluid dynamics.
\newblock {\em SIAM J. Numer. Anal.}, 40(2):492--515, 2002.

\bibitem{kunisch2004hjb}
K.~Kunisch, S.~Volkwein, and L.~Xie.
\newblock H{JB}-{POD}-based feedback design for the optimal control of
  evolution problems.
\newblock {\em SIAM J. Appl. Dyn. Syst.}, 3(4):701--722, 2004.

\bibitem{LeeTran05}
C.~H. Lee and H.~T. Tran.
\newblock Reduced-order-based feedback control of the {K}uramoto-{S}ivashinsky
  equation.
\newblock {\em J. Comput. Appl. Math.}, 173(1):1--19, 2005.

\bibitem{leibfritz2006reduced}
F.~Leibfritz and S.~Volkwein.
\newblock Reduced order output feedback control design for {PDE} systems using
  proper orthogonal decomposition and nonlinear semidefinite programming.
\newblock {\em Linear Algebra Appl.}, 415(2-3):542--575, 2006.

\bibitem{luo2008mixed}
Zhendong Luo, Jing Chen, I.~M. Navon, and Xiaozhong Yang.
\newblock Mixed finite element formulation and error estimates based on proper
  orthogonal decomposition for the nonstationary {N}avier-{S}tokes equations.
\newblock {\em SIAM J. Numer. Anal.}, 47(1):1--19, 2008/09.

\bibitem{MohebujjamanRebholzXieIliescu17}
Muhammad Mohebujjaman, Leo~G. Rebholz, Xuping Xie, and Traian Iliescu.
\newblock Energy balance and mass conservation in reduced order models of fluid
  flows.
\newblock {\em J. Comput. Phys.}, 346:262--277, 2017.

\bibitem{Muralikrishnan17}
Sriramkrishnan Muralikrishnan, Minh-Binh Tran, and Tan Bui-Thanh.
\newblock i{HDG}: {A}n {I}terative {HDG} {F}ramework for {P}artial
  {D}ifferential {E}quations.
\newblock {\em SIAM J. Sci. Comput.}, 39(5):S782--S808, 2017.

\bibitem{nguyen2012hybridizable}
N.~C. Nguyen and J.~Peraire.
\newblock Hybridizable discontinuous {G}alerkin methods for partial
  differential equations in continuum mechanics.
\newblock {\em J. Comput. Phys.}, 231(18):5955--5988, 2012.

\bibitem{nguyen2009implicit1}
N.~C. Nguyen, J.~Peraire, and B.~Cockburn.
\newblock An implicit high-order hybridizable discontinuous {G}alerkin method
  for linear convection-diffusion equations.
\newblock {\em J. Comput. Phys.}, 228(9):3232--3254, 2009.

\bibitem{nguyen2009implicit2}
N.~C. Nguyen, J.~Peraire, and B.~Cockburn.
\newblock An implicit high-order hybridizable discontinuous {G}alerkin method
  for nonlinear convection-diffusion equations.
\newblock {\em J. Comput. Phys.}, 228(23):8841--8855, 2009.

\bibitem{nguyen2010hybridizable}
N.~C. Nguyen, J.~Peraire, and B.~Cockburn.
\newblock A hybridizable discontinuous {G}alerkin method for {S}tokes flow.
\newblock {\em Comput. Methods Appl. Mech. Engrg.}, 199(9-12):582--597, 2010.

\bibitem{nguyen2011high}
N.~C. Nguyen, J.~Peraire, and B.~Cockburn.
\newblock High-order implicit hybridizable discontinuous {G}alerkin methods for
  acoustics and elastodynamics.
\newblock {\em J. Comput. Phys.}, 230(10):3695--3718, 2011.

\bibitem{nguyen2011hybridizable}
N.~C. Nguyen, J.~Peraire, and B.~Cockburn.
\newblock Hybridizable discontinuous {G}alerkin methods for the time-harmonic
  {M}axwell's equations.
\newblock {\em J. Comput. Phys.}, 230(19):7151--7175, 2011.

\bibitem{nguyen2011implicit}
N.~C. Nguyen, J.~Peraire, and B.~Cockburn.
\newblock An implicit high-order hybridizable discontinuous {G}alerkin method
  for the incompressible {N}avier-{S}tokes equations.
\newblock {\em J. Comput. Phys.}, 230(4):1147--1170, 2011.

\bibitem{Oikawa16}
Issei Oikawa.
\newblock Analysis of a reduced-order {HDG} method for the {S}tokes equations.
\newblock {\em J. Sci. Comput.}, 67(2):475--492, 2016.

\bibitem{Oxberry17}
Geoffrey~M. Oxberry, Tanya Kostova-Vassilevska, William Arrighi, and Kyle
  Chand.
\newblock {L}imited-memory adaptive snapshot selection for proper orthogonal
  decomposition.
\newblock {\em International Journal for Numerical Methods in Engineering},
  109(2):198--217, 2017.

\bibitem{ravindran2000reduced}
S.~S. Ravindran.
\newblock Reduced-order adaptive controllers for fluid flows using {POD}.
\newblock {\em J. Sci. Comput.}, 15(4):457--478, 2000.

\bibitem{rhebergen2012space}
Sander Rhebergen and Bernardo Cockburn.
\newblock A space-time hybridizable discontinuous {G}alerkin method for
  incompressible flows on deforming domains.
\newblock {\em J. Comput. Phys.}, 231(11):4185--4204, 2012.

\bibitem{riviere2008discontinuous}
B\'eatrice Rivi\`ere.
\newblock {\em Discontinuous {G}alerkin methods for solving elliptic and
  parabolic equations}, volume~35 of {\em Frontiers in Applied Mathematics}.
\newblock Society for Industrial and Applied Mathematics (SIAM), Philadelphia,
  PA, 2008.
\newblock Theory and implementation.

\bibitem{sachs2010pod}
Ekkehard~W. Sachs and Stefan Volkwein.
\newblock P{OD}-{G}alerkin approximations in {PDE}-constrained optimization.
\newblock {\em GAMM-Mitt.}, 33(2):194--208, 2010.

\bibitem{singler2014new}
John~R. Singler.
\newblock New {POD} error expressions, error bounds, and asymptotic results for
  reduced order models of parabolic {PDE}s.
\newblock {\em SIAM J. Numer. Anal.}, 52(2):852--876, 2014.

\bibitem{ueckermann2016hybridizable}
M.~P. Ueckermann and P.~F.~J. Lermusiaux.
\newblock Hybridizable discontinuous {G}alerkin projection methods for
  {N}avier-{S}tokes and {B}oussinesq equations.
\newblock {\em J. Comput. Phys.}, 306:390--421, 2016.

\bibitem{volkwein2013proper}
Stefan Volkwein.
\newblock Proper orthogonal decomposition: {T}heory and reduced-order modelling
  (lecture notes), 2013.

\bibitem{XieWellsWangIliescu18}
Xuping Xie, David Wells, Zhu Wang, and Traian Iliescu.
\newblock Numerical analysis of the {L}eray reduced order model.
\newblock {\em J. Comput. Appl. Math.}, 328:12--29, 2018.

\end{thebibliography}

\end{document}